\documentclass[letterpaper,12pt]{article}
\usepackage{amsmath,amssymb,amsfonts,epsf,epsfig,enumerate,xcolor,placeins}

\newtheorem{thm}{Theorem}[section]

\newtheorem{lem}[thm]{Lemma}
\newtheorem{defn}[thm]{Definition}
\newtheorem{alg}[thm]{Algorithm}
\newtheorem{cor}[thm]{Corollary}
\newtheorem{eg}[thm]{Example}

\newtheorem{rems}[thm]{Remarks}

\newenvironment{proof}{\noindent {\sc Proof}.}
                {\phantom{a} \hfill \framebox[2.2mm]{ } \bigskip}

\newcommand{\ZZ}{\mathbb{Z}}

\def\int{{\rm int}}

\def\sm{{\backslash}}
\newcommand{\F}{{\mathcal{F}}}

\renewcommand{\phi}{\varphi}

\newcommand\mset[1]{\left\{\!\!\left\{#1\right\}\!\!\right\}} 

\title{Using edge cuts \\ to find Euler tours and Euler families in hypergraphs}
\author{Mateja \v{S}ajna\footnote{Corresponding author. Email: msajna@uottawa.ca. Mailing address: Department of Mathematics and Statistics, University of Ottawa, 150 Louis-Pasteur Private, Ottawa, ON, K1N 6N5, Canada.} \; and Andrew Wagner \\
{\small University of Ottawa}}

\begin{document}
\maketitle \baselineskip 17pt

\begin{abstract}
An {\em Euler tour} in a hypergraph is a closed walk that traverses each edge of the hypergraph exactly once, while an {\em Euler family} is a family of closed walks that jointly traverse each edge exactly once and cannot be concatenated. In this paper, we show how the problem of existence of an Euler tour (family) in a hypergraph $H$ can be reduced to the analogous problem in some smaller hypergraphs that are derived from $H$ using an edge cut of $H$. In the process, new techniques of {\em edge cut assignments} and {\em collapsed hypergraphs} are introduced. Moreover, we describe algorithms based on these characterizations that determine whether or not a hypergraph admits an Euler tour (family), and can also  construct an Euler tour (family) if it exists.

\end{abstract}

\section{Introduction}

It is common knowledge that a connected graph admits an Euler tour --- that is, a closed walk traversing each edge of the graph exactly once --- if and only if  the graph has no vertices of odd degree. The notion of an Euler tour can be generalized to hypergraphs in the obvious way, and has been studied as such in \cite{LonNar,BahSaj-17}. In addition, Bahmanian and \v{S}ajna \cite{BahSaj-17} also introduced the notion of an {\em Euler family}, which is a family of closed walks that jointly traverse each edge of the hypergraph exactly once and cannot be concatenated. For a connected graph, an Euler family precisely corresponds to an Euler tour; for general connected hypergraphs, however, the two notions give rise to two rather distinct problems: {\sc Euler family}, which is of polynomial complexity \cite{BahSaj-17}, and {\sc Euler tour}, which is NP-complete \cite{LonNar}. The question of how to reduce the search for an Euler family or tour in a hypergraph to smaller hypergraphs is thus particularly pertinent in the case of Euler tours.

An Euler tour (family) is called {\em spanning} if it traverses every vertex of the hypergraph. In \cite{SteSaj}, Steimle and \v{S}ajna showed how to use certain vertex cuts to reduce the problem of finding a spanning Euler tour (family) in a hypergraph $H$ to some smaller hypergraphs derived from subhypergraphs of $H$. In the present paper, with the same goal in mind, we shall use edge cuts instead. This new approach represents a major improvement over the results from \cite{SteSaj}. First, it can be used for general Euler tours and families, not just for spanning Euler tours and families. Second, while the main results from \cite{SteSaj} apply only to hypergraphs with vertex cuts of cardinality at most two, the present approach applies to edge cuts of any size, and hence to all hypergraphs, as every non-trivial hypergraph has an edge cut. In addition, we introduce new elegant techniques of edge cut assignments and collapsed hypergraphs, which will be useful in problems beyond the scope of this paper.

After reviewing the terminology and providing some basic tools in Section~\ref{sec:prelim}, we shall introduce edge cuts in Section~\ref{sec:edge_cuts}, and the first main tool --- edge cut assignments --- in Section~\ref{sec:alpha}. The key result of this section is Theorem~\ref{thm:H^alpha}, where we show that a hypergraph admits an Euler family (tour) if and only if there exists an edge cut assignment such that the associated hypergraph admits an Euler family (tour). This theorem  forms the basis of all reduction results and algorithms.

In Section~\ref{sec:red1}, we present our first main reduction result. Namely, in Theorems~\ref{thm:EF-J} and \ref{thm:ET-J} we show that a hypergraph $H$ with a minimal edge cut $F$ admits an Euler family (tour) if and only if certain subhypergraphs of H (obtained from the connected components of $H \sm F$) admit an Euler family (tour).

We begin Section~\ref{sec:red2} by introducing a collapsed hypergraph of a hypegraph $H$; that is, a hypergraph obtained from $H$ by ``collapsing'' (identifying) the vertices in a given subset of the vertex set of $H$. We then present our second main reduction result. In Theorem~\ref{thm:EF-collapsed}, we show that a hypergraph $H$ admits an Euler family if and only if certain collapsed hypergraphs obtained via an edge cut assignment admit Euler families. The analogous result for Euler tours is presented in Corollaries~\ref{cor:ET-collapsed-1} and \ref{cor:ET-collapsed2}; the former contains the proof of necessity, while the latter contains the proof of sufficiency, but requires an additional assumption.

Each of Sections~\ref{sec:red1} and \ref{sec:red2} concludes with a description of pertinent algorithms that determine whether or not a hypergraph admits an Euler family (tour). Since all of our proofs are constructive, the algorithms can easily be modified to construct an Euler family (tour) if one exists.

\section{Preliminaries}\label{sec:prelim}

We begin with some basic concepts related to hypergraphs, which will be used in later discussions. For any graph- and hypergraph-theoretic terms not defined here, we refer the reader to \cite{BonMur} and \cite{BahSaj-15}, respectively.

A {\em hypergraph} $H$ is an ordered pair $(V,E)$, where $V$ is a non-empty finite set and $E$ is a finite multiset of $2^V$. To denote multisets, we shall use double braces, $\mset{.}$. The elements of $V = V(H)$ and $E = E(H)$ are called  {\em vertices} and  {\em edges}, respectively. A hypergraph is said to be {\em trivial} if it has only one vertex, and {\em empty} if it has no edges. \textcolor{black}{The {\em size} of a hypergraph $H=(V,E)$ is $|H|=\sum_{e \in E} |e|$.}

Let $H = (V,E)$ be a hypergraph, and $u,v\in V$. If $u \ne v$ and there exists an edge $e\in E$ such that $u,v\in e$, then we say that $u$ and $v$ are {\em adjacent (via the edge $e$)}; this is denoted $u \sim_H v$.
If $v\in V$ and  $e\in E$ are such that $v\in e$, then $v$ is said to be {\em incident} with $e$, and the ordered pair $(v,e)$ is called a {\em flag} of $H$. The set of flags of $H$ is denoted by $F(H)$.
The {\em degree} of a vertex $v\in V$ is the number of edges in $E$ incident with $v$, and is denoted by $\deg_H(v)$, or simply $\deg(v)$ when there is no ambiguity.

A hypergraph $H' = (V',E')$ is called a {\em subhypergraph} of the hypergraph $H = (V,E)$ if $V'\subseteq V$ and $E'=\mset{ e \cap V': e \in E''}$ for some submultiset $E''$ of $E$. For any subset $V'\subseteq V$, we define the {\em subhypergraph of $H$ induced by} $V'$ to be the hypergraph $(V',E')$ with $E' = \mset{e\cap V': e\in E, e\cap V'\ne\emptyset}$. Thus, we obtain the subhypergraph induced by $V'$  by deleting all vertices in $V - V'$ from $V$ and from each edge of $H$, and subsequently deleting all empty edges.
For any subset $E'\subseteq E$, we denote the subhypergraph $(V, E - E')$ of $H$ by $H \sm E'$, and for $e \in E$, we write $H  \sm  e$ instead of $H \sm  \{ e\}$. For any multiset $E'$ of $2^V$, the symbol $H+E'$ will denote the hypergraph obtained from $H$ by adjoining all edges in $E'$.

A $(v_0,v_k)$-{\em walk of length} $k$ in a hypergraph $H$ is an alternating sequence $W = v_0e_1v_1\ldots$ $ v_{k-1}e_kv_k$ of (possibly repeated) vertices and edges such that $v_0,\ldots,v_k\in V$, $e_1,\ldots,e_k\in E$, and for each $i\in\lbrace 1,\ldots,k\rbrace$,
the vertices $v_{i-1}$ and $v_i$ are adjacent in $H$ via the edge $e_i$.
Note that since adjacent vertices are by definition distinct, no two consecutive vertices in a walk can be the same. It follows that no walk in a hypergraph contains an edge of cardinality less than 2. The vertices in $V_a(W) = \lbrace v_0,\ldots,v_k\rbrace$ are called the {\em anchors} of $W$, vertices $v_0$ and $v_k$ are the {\em endpoints} of $W$, and $v_1,\ldots,v_{k-1}$ are the {\em internal vertices} of $W$. We also define the edge set of $W$ as $E(W)=\lbrace e_1,\ldots,e_k\rbrace$, and the set of {\em anchor flags} of $W$ as $F(W)=\{ (v_0,e_1), (v_1,e_1), (v_2, e_2), \ldots, (v_{k-1},e_k), (v_k, e_k) \}$.
Walks $W$ and $W'$ in a hypergraph $H$ are said to be {\em edge-disjoint} if $E(W)\cap E(W') = \emptyset$, and {\em anchor-disjoint} if $V_a(W)\cap V_a(W') = \emptyset$.

A walk $W = v_0e_1v_1\ldots v_{k-1}e_kv_k$  is called {\em closed} if $v_0 = v_k$ and $k \ge 2$;  a {\em trail} if the edges $e_1,\ldots,e_k$ are pairwise distinct; a {\em path} if it is a trail and the vertices $v_0,\ldots,v_k$ are pairwise distinct; and a {\em cycle} if it is a closed  trail and the vertices $v_0,\ldots,v_{k-1}$ are pairwise distinct. (Note that in \cite{BahSaj-15}, a ``trail'' was defined as a walk with no repeated anchor flags, and a walk with no repeated edges was called a ``strict trail''. In this paper, we shall consider only strict trails, and hence use the shorter term ``trail'' to mean a ``strict trail''.)

A walk $W = v_0e_1v_1\ldots v_{k-1}e_kv_k$ is said to {\em traverse} a vertex $v$ and edge $e$ if $v \in V_a(W)$ and $e \in E(W)$, respectively. More specifically, the walk $W$ is said to {\em traverse} the edge $e$ {\em via vertex} $v$, as well as {\em traverse} vertex $v$ {\em via edge} $e$, if either $ev$ or $ve$ is a subsequence of $W$.

Vertices $u$ and $v$ are {\em connected} in a hypergraph $H$ if there exists a $(u,v)$-walk  --- or equivalently, a $(u,v)$-path \cite[Lemma 3.9]{BahSaj-15} --- in $H$, and $H$ itself is {\em connected} if every pair of vertices in $V$ are connected in $H$. The {\em connected components} of $H$ are the maximal connected subhypergraphs of $H$ without empty edges. The number of connected components of $H$ is denoted by $c(H)$.

An {\em Euler tour} of a hypergraph $H$ is a closed trail of $H$ traversing every edge of $H$, and a hypergraph is said to be {\em eulerian} if it either admits an Euler tour or is trivial and empty. An {\em Euler family} of $H$ is a set of pairwise edge-disjoint and anchor-disjoint closed trails of $H$ jointly traversing every edge of $H$. Clearly, an Euler family of cardinality 1 corresponds to an Euler tour, and vice-versa.
An Euler family $\cal F$ of $H$ is said to be {\em spanning} if every vertex of $H$ is an anchor of a closed trail in $\cal F$, and an Euler tour of $H$ is {\em spanning} if it traverses every vertex of $H$.

The following two observations will be frequently used tools in the rest of the paper. Their easy proofs are omitted.

\begin{lem}\label{lem:decomposingfamilies}
Let $H$ be a hypergraph with connected components $H_i,$ for $i\in I$.  Then the following hold.
\begin{enumerate}[(i)]
\item If, for each $i \in I$, we have that $H_i$ has an Euler family $\F_i$, then $\bigcup_{i\in I} \F_i$ is an Euler family of $H$.  If each $\F_i$ is spanning in $H_i$, then $\F$ is spanning in $H$.
\item If $H$ has an Euler family $\F$, then $\F$ has a partition $\{ \F_i: i\in I \}$ such that, for each $i\in I$, we have that $\F_i$ is an Euler family of $H_i$.  If $\F$ is spanning in $H$, then   $\F_i$ is spanning in $H_i$, for each $i \in I$.
\end{enumerate}
\end{lem}

\begin{lem}\label{lem:subhypergraph}
Let $H_1$ and $H_2$ be hypergraphs such that $V(H_1) \subseteq V(H_2)$ and there exists a bijection $\phi: E(H_1) \to E(H_2)$ satisfying $e \subseteq \phi(e)$ for all $e \in E(H_1)$. Then the following hold.
\begin{enumerate}[(i)]
\item If $H_1$ has an Euler family $\F_1$, then $H_2$ has an Euler family $\F_2$ obtained from $\F_1$ by replacing each edge $e$ with $\phi(e)$.
\item If $H_2$ has an Euler family $\F_2$ such that for all $e \in E(H_2)$ we have that $e$ is traversed in $\F_2$ via vertices in $\phi^{-1}(e)$, then $H_1$ has an Euler family $\F_1$ obtained from $\F_2$ by replacing each edge $e$ with $\phi^{-1}(e)$.
\end{enumerate}
\end{lem}

\color{black}
Lemma~\ref{lem:subhypergraph} also gives rise to the following tool, to be used in our algorithms.

\begin{cor}\label{cor:2-edges}
Let $H=(V,E)$ be a hypergraph, and $G$ an even graph with $V(G) \subseteq V$ and $E(G) \subseteq E$. Then the following hold.
\begin{enumerate}[(i)]
\item $H$ has an Euler family if and only if $H \sm E(G)$ does.
\item Assume $G$ is connected and $C$ is a cycle (graph) on the vertex set $V(G)$. Then $H$ has an Euler tour if and only if $\left( H \sm E(G) \right)+E(C)$ does.
\end{enumerate}
\end{cor}

\begin{proof}
\begin{enumerate}[(i)]
\item Assume $H$ has an Euler family $\F$. Let $G_{\F}$ be the graph with vertex set $V$ and $E(G_{\F})=\mset{ \psi(e): e \in E}$, where $\psi(e)=\{ x,y \}$ if $\F$ traverses $e$ via its vertices $x$ and $y$. By Lemma~\ref{lem:subhypergraph}(ii), the graph $G_{\F}$ has an Euler family, and hence is even. Therefore $G_{\F} \sm E(G)$ is even and possesses an Euler family, and hence by Lemma~\ref{lem:subhypergraph}(i), the hypergraph $H \sm E(G)$ possesses an Euler family.

    Conversely, assume that $H \sm E(G)$ admits an Euler family $\F_1$, and let $\F_G$ be an Euler family of $G$. Then an Euler family of $H$ is obtained from $\F_1 \cup \F_G$ is by concatenating any closed trails with a common anchor.

\item Let $H'=\left( H \sm E(G) \right)+E(C)$.

Assume $H$ has an Euler tour $T$. Let $G_{T}$ be the graph with vertex set $V$ and $E(G_{T})=\mset{ \psi(e): e \in E}$, where $\psi(e)=\{ x,y \}$ if $T$ traverses $e$ via its vertices $x$ and $y$. By Lemma~\ref{lem:subhypergraph}(ii), the graph $G_{T}$ has an Euler tour, and hence is even and connected. Let $G_T'=\left(G_T \sm E(G)\right) + E(C)$. Then $G_T'$ is a connected even graph, and hence admits an Euler tour. By Lemma~\ref{lem:subhypergraph}(i), the hypergraph $H'$ admits an Euler tour.

The converse is obtained similarly, interchanging the roles of $H$ and $H'$, and those of $G$ and $C$.
\end{enumerate}
\end{proof}

\color{black}

\section{Edge cuts in hypergraphs}\label{sec:edge_cuts}

As in graphs \cite{BonMur}, an {\em edge cut} in a hypergraph $H=(V,E)$ is a (multi)set of edges of the form $[S,V-S]_H$ for some non-empty proper subset $S$ of $V$. Here, for any $S, T \subseteq V$, we denote
$$[S,T]_H= \mset{ e \in E: e \cap S \ne \emptyset \mbox{ and } e \cap T \ne \emptyset}.$$
An edge $e$ in a hypergraph $H$ is said to be a {\em cut edge} of $H$ if $c(H \sm e)>c(H)$. Thus, an edge $e$ is a cut edge if and only if $\{ e \}$ is an edge cut.

\begin{lem}\label{lem:minedgecut}
Let $H=(V,E)$ be a hypergraph and $F\subseteq E.$ Then $H \sm F$ is disconnected if and only if $H$ has an edge cut $F' \subseteq F$.
\end{lem}

\begin{proof}
Assume $H \sm F$ is disconnected, let $H_1$ be one connected component of $H \sm F$, and $S=V(H_1)$. Then $F'=[S,V-S]_H$ is an edge cut of $H$ contained in $F$.

Conversely, assume $H$ has an edge cut $F' \subseteq F$, where $F'=[S,V-S]_H$ for $\emptyset \subsetneq S \subsetneq V$. Then $H \sm F$ has no edge intersecting both $S$ and $V-S$, so it is disconnected.
\end{proof}

As we shall see in the next lemma, minimal edge cuts --- that is, edge cuts that are not properly contained in other edge cuts --- have a very nice property that will be much exploited in subsequent results. Note that, for a hypergraph $H=(V,E)$ and its subhypergraph $H'$, we say that an edge $e \in E$ {\em intersects} $H'$ if $e \cap V(H')\ne \emptyset$.

\begin{lem}\label{lem:intersects}
Let $H=(V,E)$ be a hypergraph. An edge cut $F$ of $H$ is minimal if and only if every edge of $F$ intersects each connected component of $H \sm F$.
\end{lem}

\begin{proof}
Assume $F$ is a minimal edge cut of $H$. Suppose $H_1$ is a connected component of $H \sm F$, and $f \in F$ is such that $f$ does not intersect $H_1$. Let $S=V(H_1)$ and $F'=F-\{ f \}$. Then $[S,V-S]_H \subseteq F' \subsetneq F$, contradicting the minimality of $F$. Thus every edge of $F$ intersects each connected component of $H \sm F$.

Conversely, assume each edge of $F$ intersects every connected component of $H\sm F$.  Suppose $F'$ is an edge cut of $H$ and $F' \subsetneq F$. Take any $f \in F-F'$. Since $f$ intersects every connected component of $H\sm F$, the hypergraph $H+f=H \sm F'$ is connected. Hence by Lemma~\ref{lem:minedgecut}, the set $F'$ contains no edge cut of $H$, a contradiction. Hence the edge cut $F$ is minimal.
\end{proof}

\color{black}

In our algorithms, it will be helpful to find minimal edge cuts containing edges of cardinality at least 3. The proof of the next lemma explains how this can be done.

\begin{lem}\label{lem:min-cut-large-edge}
Let $H=(V,E)$ be a connected hypergraph, and $e^\ast$ an edge of $H$ of cardinality at least 3. Then $H$ admits a minimal edge cut $F^\ast$ such that $e^\ast \in F^\ast$ or $F^\ast$ has no edges of cardinality less than 3.
\end{lem}

\begin{proof}
Fix a vertex $x \in e^\ast$, and let $F=\mset{ e\in E: x \in e}$. Then $F$ is an edge cut; let $H_1,\ldots,H_k$ be the connected components of $H \sm F$. Without loss of generality, we may assume that $e^\ast$ intersects $H_1$.

Let $F_1=\mset{ e \in F: e \cap V(H_1) \ne \emptyset }$, so $e^\ast \in F_1$. Since vertices of $H_1$ are disconnected from vertex $x$ in $H \sm F_1$, by Lemma~\ref{lem:minedgecut}, there exists a minimal edge cut $F^\ast \subseteq F_1$.

Let $H_0^\ast,H_1^\ast,\ldots,H_{\ell}^\ast$ be the connected components of $H \sm F^\ast$. Observe that, without loss of generality, we may assume that $x \in V(H_0)^\ast$, while $H_1^\ast,\ldots,H_{\ell}^\ast \in \{ H_1,\ldots, H_k \}$.

\color{black}

If $\ell \ge 2$, then by Lemma~\ref{lem:intersects}, each edge of $F^\ast$ has cardinality at least 3. Hence we may assume that $\ell=1$. If $H_1^\ast=H_1$, then $e^\ast \in F^\ast$, and we are done. Otherwise, each edge in $\F^\ast$ contains $x$, as well as a vertex in $H_1$ and a vertex in $H_1^\ast$, and hence has cardinality at least 3.

We conclude that  $e^\ast \in F^\ast$ or each edge of $F^\ast$ has cardinality at least 3, as claimed.
\end{proof}

\color{black}

\section{Edge Cut Assignments}\label{sec:alpha}

We are now ready to present the first of the two main tools that we shall use to study eulerian properties of hypergraphs via their edge cuts; namely, an edge cut assignment $\alpha$ associated with an edge cut $F$ of a hypergraph $H$, which maps each edge of $F$ to an unordered pair of connected components (or, more generally, unions of connected components) of $H \sm F$.

Note that, for a finite set $I$, we denote the set of all unordered pairs of elements in $I$, with repetitions in a pair permitted, by $I^{[2]}$; that is, $I^{[2]}=\{ ij: i,j \in I \}$. Thus, $|I^{[2]}|=\frac{1}{2}n(n+1)$.

\begin{defn}\label{defn:G}{\rm
Let $H$ be a connected hypergraph with a minimal edge cut $F$, and let $\{ V_i: i \in I\}$ be a partition of $V(H)$ such that each $V_i$ is a union of the vertex sets of the connected components of $H \sm F$.
\begin{itemize}
\item A mapping $\alpha: F \to I^{[2]}$ is called an {\em edge cut assignment} (associated with $F$ and $H$) if, for all $f \in F$, we have that $\alpha(f)=ij$ implies that $f \cap V_i \ne \emptyset$ and $f \cap V_j \ne \emptyset$, and $\alpha(f)=i$ implies that $|f \cap V_i| \ge 2$.

\item An edge cut assignment $\alpha: F \to I^{[2]}$ is called {\em standard} if $V_i$, for each $i \in I$, is the vertex set of a connected component of $H \sm F$.

\item Given an edge cut assignment $\alpha$, we define, for each $e \in E(H)$,
$$
e^{\alpha}=\left\{
\begin{array}{ll}
e \cap \left( V_i \cup V_j \right) & \mbox{ if } e \in F \mbox{ and } \alpha(e)=ij \\
e & \mbox{ if } e \not\in F
\end{array}
\right..
$$
\item A hypergraph $H^{\alpha}$, defined by
$$V(H^{\alpha})=V(H) \quad \mbox{and} \quad E(H^{\alpha})=\mset{ e^{\alpha}: e \in E(H) },$$
is called the {\em hypergraph  associated with $H$ and the edge cut assignment $\alpha$}.
\item A multigraph $G^{\alpha}$ with vertex set $V(G^{\alpha})=I$ and edge multiset $E(G^{\alpha})=\mset{ \alpha(f): f \in F}$ is called the {\em multigraph associated with the edge cut assignment $\alpha$}. Thus
    $\alpha$ can be viewed as a bijection $F \to E(G^{\alpha})$.
\end{itemize}
}
\end{defn}

The usefulness of these concepts will be conveyed in the following three observations, the last of which yields necessary and sufficient conditions for a hypergraph to admit an Euler family (tour) via an edge cut assignment $\alpha$ and the associated hypergraph $H^{\alpha}$.

\begin{lem}\label{lem:HG^alpha}
Let $H$ be a connected hypergraph with a minimal edge cut $F$, and let $\{ V_i: i \in I\}$ be a partition of $V(H)$ into unions of the vertex sets of the connected components of $H \sm F$. Furthermore, let $\alpha: F \to I^{[2]}$ be an edge cut assignment.
If $H^{\alpha}$ has an Euler family (tour), then so does $G^{\alpha}$.
\end{lem}

\begin{proof}
Let $\F$ be an Euler family of $H^{\alpha}$, and $\F'$ the subset of $\F$ consisting of all closed trails that traverse at least on edge of $F$. Take any closed trail $T \in \F'$.   Then, without loss of generality, $T$ is of the form $T=v_0 T_0 v_0' f_0^{\alpha} v_1 T_1 v_1' f_1^{\alpha} v_2 \ldots v_{k-1} T_{k-1} v_{k-1}' f_{k-1}^{\alpha} v_0$ for some $f_0,\ldots,f_{k-1} \in F$ and $(v_i,v_i')$-trails $T_i$ in $H_i$, for $i \in \{ 0, \ldots, k-1 \}$. Obtain the sequence $T^{\alpha}$ from $T$ by replacing each subsequence $v_i T_i v_i'$ with $i$, and each edge $f_i^{\alpha}$ with $\alpha(f_i)$. Then $T^{\alpha}$ is a closed trail in $G^{\alpha}$, and $\{ T^{\alpha}: T \in \F' \}$ is an Euler family of $G^{\alpha}$.

Moreover, if $T$ is an Euler tour of $H^{\alpha}$, then $T^{\alpha}$ is Euler tour of $G^{\alpha}$.
\end{proof}

\begin{lem}\label{lem:H^alpha}
Let $H$ be a connected hypergraph with a minimal edge cut $F$, and let $\{ V_i: i \in I\}$ be a partition of $V(H)$ into unions of the vertex sets of the connected components of $H \sm F$.
\begin{enumerate}[(i)]
\item Suppose $H$ has an Euler family $\F$. Let $\alpha: F \to I^{[2]}$ be an edge cut assignment defined by $\alpha(f)=ij$ if the edge $f \in F$ is traversed by a trail in $\F$ via a vertex in $V_i$ and a vertex in $V_j$ (where $i=j$ is possible). Then  the hypergraph $H^{\alpha}$ has an Euler family obtained from $\F$ by replacing each $f \in F$ with $f^{\alpha}$.
\item Conversely, if for some edge cut assignment $\alpha: F \to I^{[2]}$, the associated hypergraph $H^{\alpha}$ has an Euler family $\F^{\alpha}$, then $H$ has an Euler family obtained from $\F^{\alpha}$ by replacing each $f^{\alpha}$, for $f \in F$, with $f$.
\end{enumerate}
\end{lem}

\begin{proof}
\begin{enumerate}[(i)]
\item Define a bijection $\phi: E(H^{\alpha}) \to E(H)$ by $\phi(e^{\alpha})=e$, for all $e \in E(H)$. Then $e^{\alpha} \subseteq  \phi(e^{\alpha}) $ for all $e^{\alpha} \in E(H^{\alpha})$. By Lemma~\ref{lem:subhypergraph}(ii), since each edge $e$ of $H$ is traversed in $\F$ via vertices in $\phi^{-1}(e)$, an Euler family of $H^{\alpha}$ is obtained from $\F$ by replacing each edge $e$ with $\phi^{-1}(e)$, which effectively means replacing each $f \in F$ with $f^{\alpha}$.
\item Define $\phi$ as in (i) and use Lemma~\ref{lem:subhypergraph}(i).
\end{enumerate}
\end{proof}

\begin{thm}\label{thm:H^alpha}
Let $H$ be a connected hypergraph with a minimal edge cut $F$, and let $\{ V_i: i \in I\}$ be a partition of $V(H)$ into unions of the vertex sets of the connected components of $H \sm F$. Then
\begin{enumerate}[(i)]
\item $H$ has an Euler family if and only if for some edge cut assignment $\alpha: F \to I^{[2]}$, each connected component of $H^{\alpha}$ has an Euler family; and
\item $H$ has an Euler tour if and only if for some edge cut assignment $\alpha: F \to I^{[2]}$, the hypergraph $H^{\alpha}$ has a unique non-empty connected component, which has an Euler tour.
\end{enumerate}
\end{thm}

\begin{proof}
Observe that, by Lemma~\ref{lem:H^alpha}, a hypergraph $H$ has an Euler family of cardinality $k$ if and only if for some edge cut assignment $\alpha$, the associated hypergraph $H^{\alpha}$ has an Euler family of cardinality $k$.
\begin{enumerate}[(i)]
\item Since by Lemma~\ref{lem:decomposingfamilies}, $H^{\alpha}$ has an Euler family if and only if each of its connected components has an Euler family, the statement follows.
\item From the above observation, $H$ has an Euler tour if and only if  $H^{\alpha}$, for some $\alpha$,  has an Euler tour. Since it is clear that $H^{\alpha}$ has an Euler tour if and only if it has a unique non-empty connected component, which itself has an Euler tour, the statement follows.
\end{enumerate}
\end{proof}

We point out that Theorem~\ref{thm:H^alpha} forms the basis for all other, more specific reductions, as well as for the algorithms presented in the rest of this paper.

\section{Reduction using standard edge cut assignments}\label{sec:red1}

In this section, we shall be using only standard edge cut assignments; that is, edge cut assignments arising from partitions of the form $\{ V(H_i): i \in I\}$, where the $H_i$, for $i \in I$, are the connected components of $H \sm F$, and $F$ is a minimal edge cut in a hypergraph $H$.

\begin{lem}\label{lem:G^alpha-conn}
Let $H$ be a connected hypergraph with a minimal edge cut $F$, and let $H_i$, for $i \in I$, be the connected components of $H \sm F$. Furthermore, let $\alpha: F \to I^{[2]}$ be a standard edge cut assignment.
\begin{enumerate}[(i)]
\item If $G'$ is a connected component of $G^{\alpha}$, then $H^{\alpha}[ \bigcup_{i \in V(G')} V(H_i)]$ is a connected component of $H^{\alpha}$.
\item If $H'$ is a connected component of $H^{\alpha}$, then $V(H')=\bigcup_{i \in J} V(H_i)$ for some $J \subseteq I$, and $G^{\alpha}[J]$ is a connected component of $G^{\alpha}$.
\end{enumerate}
\end{lem}

\newpage

\begin{proof}
\begin{enumerate}[(i)]
\item Let $G'$ be a connected component of $G^{\alpha}$, let $J=V(G')$ and $H'=H^{\alpha}[ \bigcup_{i \in J} V(H_i)]$.

    Take any $i,j \in J$ such that $i \sim_{G^{\alpha}} j$. Then there exists $f \in F$ such that $\alpha(f)=ij$ and $f^{\alpha}=f \cap (V(H_i) \cup V(H_j))$. Since $f$ intersects both $H_i$ and $H_j$, it follows that $H^{\alpha}[V(H_i) \cup V(H_j)]$ is connected. Therefore $H'$ is connected.

    Moreover, if there exist $i \in J$, $j \in I-J$, $u \in V(H_i)$, and $v \in V(H_j)$ such that $u \sim_{H^{\alpha}} v$, then for some $f \in F$ we have that $\alpha(f)=ij$, and hence $i \sim_{G^{\alpha}} j$, a contradiction.

    We conclude that $H'$ is a connected component of $H^{\alpha}$.

\item Let $H'$ be a connected component of $H^{\alpha}$. Since $H^{\alpha}=\left(\bigcup_{i \in I} H_i \right)+ \{ f^{\alpha}: f \in F \}$, there exists $J \subseteq I$ such that $V(H')=\bigcup_{i \in J} V(H_i)$.

    Take any $i,j \in J$. Then for all $u \in V(H_i)$ and $v \in V(H_j)$ we have that $u$ and $v$ are connected in $H^{\alpha}$. It is then easy to see that $i$ and $j$ are connected in $G^{\alpha}$. Hence $G^{\alpha}[J]$ is connected.

    It then follows from (i) that $G^{\alpha}[J]$ is a connected component of $G^{\alpha}$.
\end{enumerate}
\end{proof}

We are now ready to prove our first reduction theorem, Theorem~\ref{thm:EF-J}, which states that a hypergraph $H$ with a minimal edge cut $F$ admits an Euler family if and only if certain subhypergraphs of $H$ obtained from the connected components of $H \sm F$ admit Euler families. The analogous result for Euler tours follows in Theorem~\ref{thm:ET-J}.

\begin{thm}\label{thm:EF-J}
Let $H$ be a connected hypergraph with a minimal edge cut $F$, and let $H_i$, for $i \in I$, be the connected components of $H \sm F$.

Then $H$ has a (spanning) Euler family if and only if there exists $J \subseteq I$ with $1 \le |J| \le |F|$ such that
\begin{enumerate}[(i)]
\item $H\left[ \bigcup_{j \in J} V(H_j)\right]$ has a non-empty (spanning) Euler family, and
\item $H_i$ has a (spanning) Euler family for all $i \in I-J$.
\end{enumerate}
\end{thm}

\begin{proof}
Assume $H$ has an Euler family $\F$, and let $\alpha: F \to I^{[2]}$ be the edge cut assignment defined by $\alpha(f)=ij$ if the edge $f \in F$ is traversed by $\F$ via a vertex in $H_i$ and a vertex in $H_j$. Then by Lemma~\ref{lem:H^alpha}(i), the associated hypergraph $H^{\alpha}$ has an Euler family as well.  Let $G^{\alpha}$ be the associated multigraph, and $J$ the union of the vertex sets of all non-empty connected components of $G^{\alpha}$. Since $H$ is connected, we have $|F| \ge 1$, and hence $|J| \ge 1$. Since $H^{\alpha}$ has an Euler family, by Lemma~\ref{lem:HG^alpha}, so does $G^{\alpha}$. Hence
$G^{\alpha}[J]$ is an even graph with $|J|$ vertices, $|F|$ edges, and minimum degree at least 2; it follows that $|J| \le |F|$.

To show (i) and (ii), let $V'=\bigcup_{j \in J} V(H_j)$. By Lemma~\ref{lem:G^alpha-conn}(i), we have that  $H^{\alpha}[V']$ is a union of connected components of $H^{\alpha}$, and $H_i$, for each $i \in I-J$, is a connected component of $H^{\alpha}$. Since $H^{\alpha}$ has an Euler family, it follows from Lemma~\ref{lem:decomposingfamilies} that $H^{\alpha}[V']$ has an Euler family, as does $H_i$, for each $i \in I-J$.

It remains to show that $H[V']$ has a Euler family. Define a bijection
$\phi: E\left(H^{\alpha}[V']\right) \to E\left(H[ V']\right)$ by $\phi(e^{\alpha})= e \cap V'$. Then $e^{\alpha} \subseteq \phi(e^{\alpha})$ for all $e^{\alpha} \in E\left(H^{\alpha}[V']\right)$, so by Lemma~\ref{lem:subhypergraph}(i), since $H^{\alpha}[V']$ has an Euler family, so does $H[V']$. Moreover, since $H[V']$ is non-empty, this Euler family is non-empty.

Conversely, assume that there exists $J \subseteq I$ with $1 \le |J| \le |F|$ such that (i) and (ii) hold. Let $V'=\bigcup_{j \in J} V(H_j)$ and $H'=H[V'] \cup \bigcup_{i \in I-J} H_i$. Then by Lemma~\ref{lem:decomposingfamilies}(i), the hypergraph $H'$ admits an Euler family.

Observe that $E(H')=\bigcup_{i \in I}  E(H_i) \cup \{ f \cap V': f \in F \}$. Define a bijection $\phi: E(H') \to E(H)$ by $\phi(f \cap V')=f$ for all $f \in F$, and $\phi(e)=e$ for all $e \in \bigcup_{i \in I}  E(H_i)$. Hence $e \subseteq \phi(e)$ for all $e \in E(H')$, and since $H'$ has an Euler family, by Lemma~\ref{lem:subhypergraph}(i), the hyperhgraph $H$ has an Euler family as well.

It is easy to see that if the Euler family $\F$ of $H$ is spanning, then so are the resulting Euler families of $H[V']$ and $H_i$, for all $i \in I-J$, and vice-versa.
\end{proof}

\begin{thm}\label{thm:ET-J}
Let $H$ be a connected hypergraph with a minimal edge cut $F$, and let $H_i$, for $i \in I$, be the connected components of $H \sm F$.

Then $H$ has an Euler tour if and only if there exists $J \subseteq I$ with $1 \le |J| \le |F|$ such that
\begin{enumerate}[(i)]
\item $H\left[ \bigcup_{j \in J} V(H_j)\right]$ has an Euler tour, and
\item $H_i$ is empty for all $i \not\in J$.
\end{enumerate}
\end{thm}

\begin{proof}
Assume $H$ has an Euler tour $T$, and let $\alpha: F \to I^{[2]}$ be the edge cut assignment defined by $\alpha(f)=ij$ if the edge $f \in F$ is traversed by $T$ via a vertex in $H_i$ and a vertex in $H_j$.
Let $G^{\alpha}$ be the associated multigraph. By Lemmas~\ref{lem:H^alpha}(i) and \ref{lem:HG^alpha}, respectively, the hypergraph $H^{\alpha}$ and multigraph $G^{\alpha}$ both have Euler tours. Hence $G^{\alpha}$ has a unique non-empty connected component; let $J$ be its vertex set. As in the proof of Theorem~\ref{thm:EF-J}, it can be shown that $1 \le |J| \le |F|$.

Let $V'=\bigcup_{j \in J} V(H_j)$. By Lemma~\ref{lem:G^alpha-conn}(i), we have that  $H^{\alpha}[V']$ is a connected component of $H^{\alpha}$, as is $H_i$, for each $i \in I-J$. Since $H^{\alpha}$ has an Euler tour and $H^{\alpha}[V']$ is not empty, it follows that $H^{\alpha}[V']$ has an Euler tour, and $H_i$, for each $i \in I-J$, is empty.

Conversely, assume that there exists $J \subseteq I$ with $1 \le |J| \le |F|$ such that (i) and (ii) hold. Let $V'=\bigcup_{j \in J} V(H_j)$, and $H'=H[V'] \cup \bigcup_{i \in I-J} H_i$. Similarly to the proof of Theorem~\ref{thm:EF-J}, we observe that since $H[V']$ has an Euler tour and $H_i$, for each $i \in I-J$, is empty, Lemma~\ref{lem:decomposingfamilies}(i) shows that $H'$ has an Euler tour as well. By Lemma~\ref{lem:subhypergraph}(i), it follows that $H$ has an Euler tour.
\end{proof}

We observe that Theorem~\ref{thm:ET-J} cannot be extended to spanning Euler tours in any meaningful way.
The following immediate consequence of
Theorems~\ref{thm:EF-J} and \ref{thm:ET-J} gives more transparent necessary and sufficient conditions for the existence of an Euler family and Euler tour for a hypergraph with a cut edge.

\begin{cor}\label{cor:cut-edge}
Let $H$ be a connected hypergraph with a cut edge $f$.  Let $H_i$,  for $i \in I$, be the connected components of $H \sm f$.  Then the following hold.
\begin{enumerate}[(i)]
\item $H$ has a (spanning) Euler family if and only if there exists $j \in I$ such that
	\begin{itemize}
	\item $H[V(H_j)]$ has a non-empty (spanning) Euler family, and
	\item $H_i$ has a (spanning) Euler family for all $i \in I-\{ j \}$.
	\end{itemize}
\item $H$ has an Euler tour if and only if there exists $j \in I$ such that
	\begin{itemize}
	\item $H[V(H_j)]$ has an Euler tour, and
	\item $H_i$ is empty for all $i \in I-\{ j \}$.
	\end{itemize}
\item $H$ has no spanning Euler tour.
\item If $H$ has no vertex of degree 1, then $H$ has no Euler tour.
\end{enumerate}
\end{cor}

\bigskip

In Algorithm~\ref{alg:EF} below, we shall now put to use the results of Lemma~\ref{lem:HG^alpha} and Theorem~\ref{thm:H^alpha} (applied to standard edge cut assignments), \textcolor{black}{as well as Corollary~\ref{cor:2-edges} and Lemma~\ref{lem:min-cut-large-edge}}, to determine whether a given hypergraph has an Euler family.

\begin{alg}\label{alg:EF} Does a connected hypergraph $H$ admit an Euler family? {\rm
\begin{enumerate}[(1)]
\item Sequentially delete vertices of degree at most 1 from $H$ until no such vertices remain or $H$ is trivial.
If at any step $H$ has an edge of cardinality less than 2, then $H$ has no Euler family --- exit.

\color{black}

\item Let $H^{[2]}$ be the graph obtained from $H$ by deleting all edges of cardinality greater than 2, and $G$ a maximal even subgraph of $H^{[2]}$. Delete all edges of $G$ from $H$.

\item Repeat Steps (1) and (2) as needed.

\item If $H$ is a graph, then it has an Euler family if and only if it is even --- exit.

\color{black}
\item Find a minimal edge cut $F$ containing an edge of cardinality at least 3.

\color{black}

\item Let $H_i$, for $i \in I$, be the connected components of $H \sm F$.
\item For all edge cut assignments $\alpha: F \to I^{[2]}$:
\begin{enumerate}
\item If $G^{\alpha}$ has no Euler family, then $H^{\alpha}$ has no Euler family --- discard this $\alpha$.
\item For each connected component $H'$ of $H^{\alpha}$: \\ if $H'$ has no Euler family {\em (recursive call)}, discard this $\alpha$.
\item If every connected component of $H^{\alpha}$ has an Euler family, then $H$ has an Euler family --- exit.
\end{enumerate}
\item $H$ has no Euler family --- exit.
\end{enumerate}
}
\end{alg}

Similarly, in Algorithm~\ref{alg:ET} below, we use Theorem~\ref{thm:ET-J} and Corollary~\ref{cor:cut-edge}, in addition to \textcolor{black}{Corollary~\ref{cor:2-edges}, Lemmas~\ref{lem:min-cut-large-edge}} and
\ref{lem:HG^alpha}, and Theorem~\ref{thm:H^alpha},  to determine whether a given hypergraph has an Euler tour.

\begin{alg}\label{alg:ET} Does a connected hypergraph $H$ admit an Euler tour?
{\rm
\begin{enumerate}[(1)]
\item Sequentially delete vertices of degree at most 1 from $H$ until no such vertices remain or $H$ is trivial.
If at any step $H$ has an edge of cardinality less than 2, then $H$ is not eulerian --- exit.

\color{black}

\item Let $H^{[2]}$ be the graph obtained from $H$ by deleting all edges of cardinality greater than 2, and $G$ a maximal even subgraph of $H^{[2]}$. For each non-empty connected component $G_1$ of $G$, replace $E(G_1)$ in $H$ with the edge set of any cycle on $V(G_1)$.

\item Repeat Steps (1) and (2) as needed.

\item If $H$ is a graph, then it is eulerian if and only if it is even and has at most one non-trivial connected component --- exit.

\color{black}
\item Find a minimal edge cut $F$ containing an edge of cardinality at least 3. If $|F|=1$, then $H$ is not eulerian --- exit.

\color{black}

\item Let $H_i$, for $i \in I$, be the connected components of $H \sm F$.
\item Let $J = \{ i \in I: H_i \mbox{ is non-empty} \}$. If $|F|<|J|$, then $H$ is not eulerian --- exit.
\item For all edge cut assignments $\alpha: F \to I^{[2]}$:
\begin{enumerate}
\item If $G^{\alpha}$ is not eulerian, then $H^{\alpha}$ is not eulerian --- discard this $\alpha$.
\item If $H^{\alpha}$ has at least 2 non-empty connected components, then $H^{\alpha}$ is not eulerian --- discard this $\alpha$.
\item Let $H'$ be the unique non-empty connected component of $H^{\alpha}$. If $H'$ is eulerian {\em (recursive call)}, then $H$ is eulerian --- exit. Otherwise, discard this $\alpha$.
\end{enumerate}
\item $H$ is not eulerian --- exit.
\end{enumerate}
}
\end{alg}

\begin{figure}[t]
\centerline{\includegraphics[scale=0.8]{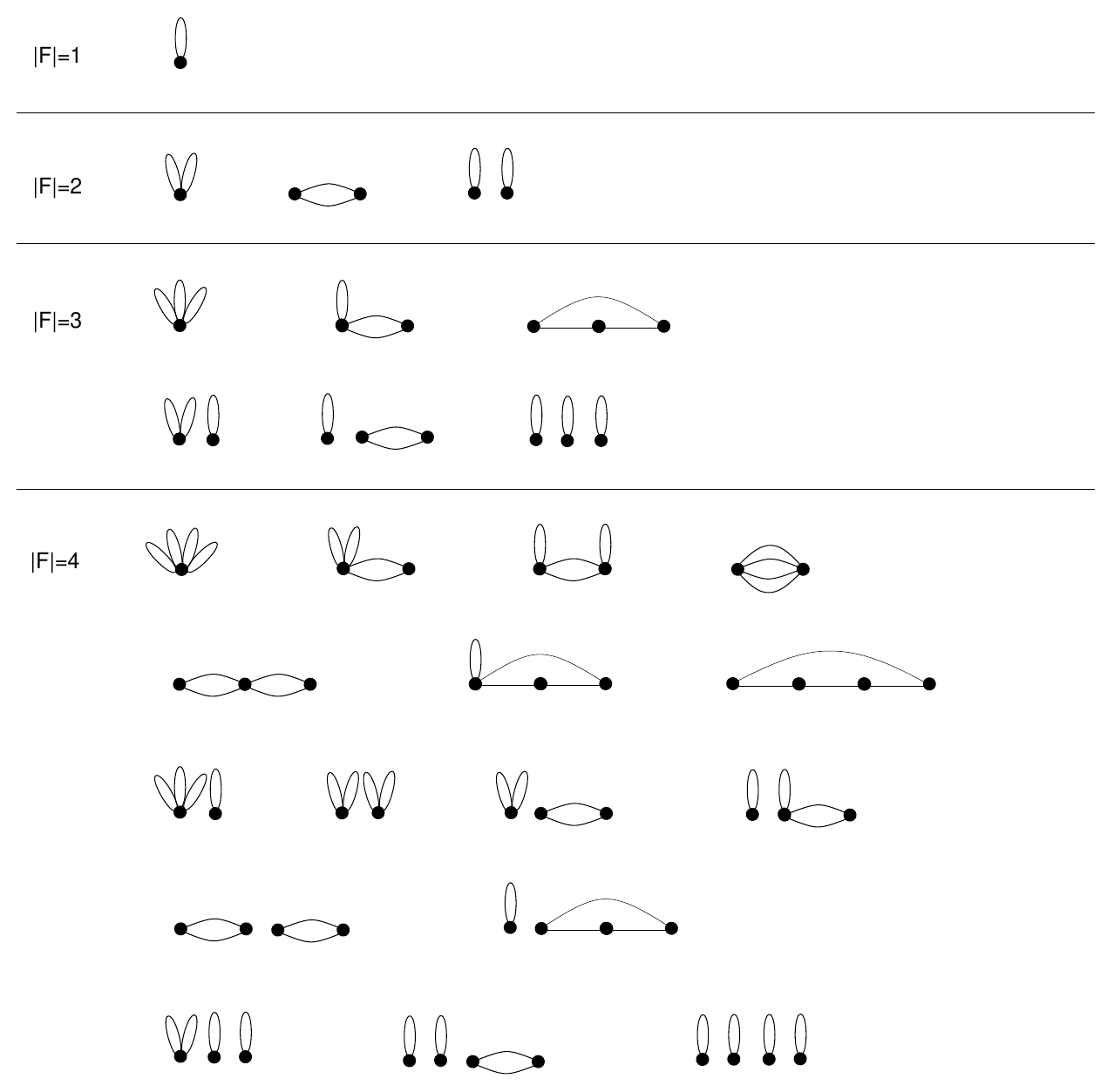}}
\caption{Isomorphism classes of even graphs of small size without isolated vertices.}\label{fig:Fig1}
\end{figure}

\begin{rems}\label{rems}{\rm
\begin{enumerate}[(i)]
\item In both algorithms, the input hypergraph is being modified during the execution of the algorithm. However, at any step, the current hypergraph $H$ admits an Euler family (tour) if and only if the input hypergraph does.

\color{black}
\item As each non-empty proper subset $S$ of $V(H)$ corresponds to an edge cut, minimal edge cuts are easy to construct. If the minimal edge cut obtained in this straightforward way does not contain an edge of cardinality at least 3, then  at Step (5) of both algorithms, the proof of Lemma~\ref{lem:min-cut-large-edge} can be used to satisfy this requirement. Its purpose is to ensure the reduction in size of the hypergraph for the next recursive call, as explained below.

\color{black}
\item \label{rem:reduction} To guarantee that the algorithms terminate, it suffices that at each iteration $|H'|<|H|$. Some adjustments may be needed to satisfy this condition. First, when generating edge cut assignments, priority should be given to $\alpha$ such that $\alpha(f)=i$ for some $i \in I$ --- for as many $f \in F$ as possible in case of Algorithm~\ref{alg:EF}, and for at least one $f \in F$ in case of Algorithm~\ref{alg:ET}. In other words, we want $H^\alpha$ to have as many connected components as possible (while allowing a positive outcome) in the hope that the reduction in size from $H$ to $H'$ is as large as possible.

    However, we may still need to examine $\alpha$ such that $H^\alpha=H$ and hence $H'=H$ and no reduction in size has occurred. In that case, it must be that $|I|=2$ --- that is, $H \sm F$ has just two connected components, $H_1$ and $H_2$ --- and $\alpha(f)=12$ for all $f \in F$. We then choose $f^\ast \in F$ such that $|f| \ge 3$. For all $x_1 \in f^\ast \cap V(H_1)$ and $x_2 \in f^\ast \cap V(H_2)$, we let $e^\ast=\{ x_1,x_2\}$ and $H'=(H \sm f^\ast) + \{ e^\ast\} $. Note that $|H'|<|H|$, and $H$ admits an Euler family (tour) if and only if at least one such hypergraph $H'$ does. If necessary, we then recursively test all such hypergraphs $H'$.

\color{black}

\item \label{rem:min} \textcolor{black}{As long as we can guarantee that the algorithms terminate,} using a minimum edge cut (instead of just a minimal one) would likely be more efficient. An ${\cal O}(p+n^2 \lambda)$ algorithm for finding a minimum edge cut in a hypergraph $H$ with size $p=\sum_{e \in E(H)} |e|$, order $n=|V(H)|$, and cardinality of a minimum edge cut $\lambda$ was described in \cite{CheXu}.

\FloatBarrier

\color{black}

\item \label{rem:n-alphas} The most time-consuming part of both algorithms is likely going to be the sequence of recursive calls to determine whether hypergraphs $H'$ have an Euler family or tour. The smaller the size of the largest of the hypergraphs $H'$, the better; however, this size is impossible to predict. Instead, we suggest minimizing the number of possible edge cut assignments $\alpha$ to be verified. The number of all possible mappings $F \to I^{[2]}$  is ${|I|+1 \choose 2}^{|F|}$, which indeed suggests choosing $F$ to be a minimum edge cut \textcolor{black}{(if possible)}.
    Among minimum edge cuts $F$, should we seek one that also minimizes $|I|$? We do not have a clear answer: on one hand, this approach would minimize the number of mappings $\alpha$ to verify; on the other hand, the expected size of the hypergraphs $H'$ for our recursive calls is inversely proportional to \textcolor{black}{the number of connected components of $H^\alpha$, which can be anywhere between 1 and $|I|$.}  --- We shall comment more on the complexity of Algorithms~\ref{alg:EF} and \ref{alg:ET} in Conclusion.

\item The number of edge cut assignments $\alpha: F \to I^{[2]}$ resulting in an even graph $G^{\alpha}$ is likely much smaller than ${|I|+1 \choose 2}^{|F|}$. To give some idea of this number, we show in Figure~\ref{fig:Fig1}, for $|F| \le 4$, the isomorphism classes of the suitable graphs $G^{\alpha}$ with the isolated vertices removed.

\item Observe that Algorithms~\ref{alg:EF} and \ref{alg:ET} are
formulated to solve decision problems --- does $H$ contain an Euler family (tour)? --- however, since they are based on results that use constructive proofs, they can easily be adapted to construct an Euler family (tour), if it exists.
\end{enumerate}
}
\end{rems}

\color{black}

\begin{figure}[t]
\centerline{\includegraphics[scale=0.8]{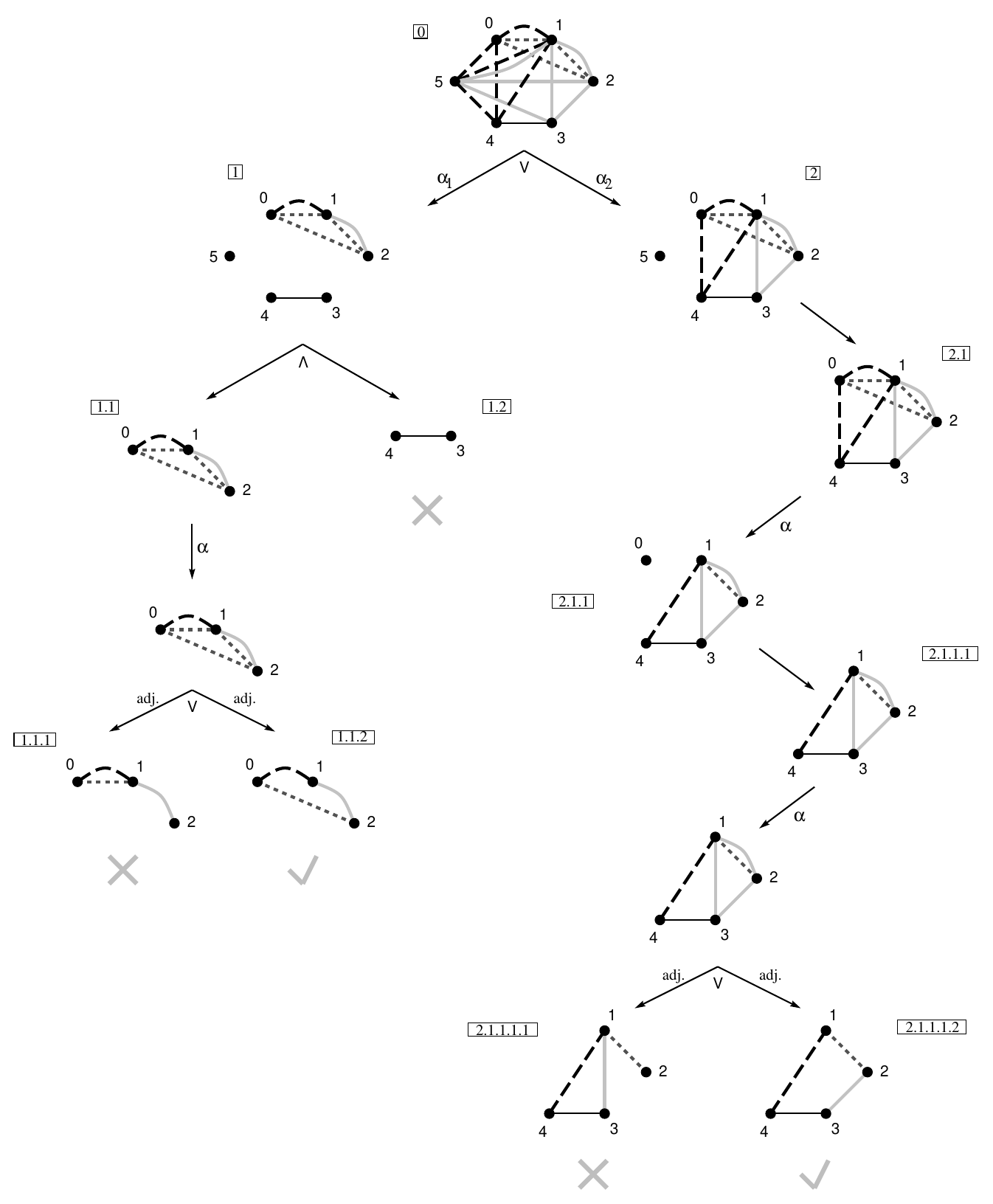}}
\caption{Illustrating Example~\ref{eg:Alg1}.}\label{fig:Fig2}
\end{figure}

\begin{eg}\label{eg:Alg1}{\rm
We shall demonstrate Algorithm~\ref{alg:EF} using the hypergraph $H=(V,E)$ with $V=\ZZ_6$ and $E=\{ 012, 0145,1235,34 \}$. The key hypergraphs in the execution of the algorithm on $H$ are shown in Figure~\ref{fig:Fig2}. In each hypergraph, each edge $e$ is drawn as a complete graph on the vertex set $e$, using a different line style. Symbols $\vee$ and $\wedge$ indicate that both or either, respectively, of the two branches must lead to success, and $\alpha$ indicates that an edge cut assignment has been applied. Finally, {\em adj.} indicates that an adjustment has been applied in order to reduce the size of the hypergraph; see Remarks~\ref{rems}(\ref{rem:reduction}). Note that some minor details are omitted for brevity.

{\em Iteration 0.} The input hypergraph is  $H=\left(\ZZ_6,\{ 012, 0145,1235,34 \} \right)$, and we choose a minimal edge cut $F=\{ 0145,1235 \}$. The hypergraph $H\sm F$ has connected components $H_1, H_2$, and $H_3$, with vertex sets $V(H_1)=\ZZ_3$, $V(H_2)=\{ 3,4 \}$, and $V(H_3)=\{ 5 \}$. We first examine the edge cut assignment $\alpha$ with $\alpha(0145)=\alpha(1235)=1$ (Iteration 1), and then the edge cut assignment $\alpha$ with $\alpha(0145)=\alpha(1235)=12$ (Iteration 2).

{\em Iteration 1.} We have $H^\alpha=\left(\ZZ_6,\{ 01, 012, 12,34 \} \right)$, and we need to test its two non-empty connected components (Iterations 1.1 and 1.2).

\FloatBarrier

{\em Iteration 1.1.} The input hypergraph is  $H=\left(\ZZ_3,\{ 01,012, 12 \} \right)$. We choose a minimal edge cut $F=\{ 01,012 \}$. The hypergraph $H\sm F$ has connected components $H_1$ and $H_2$, with vertex sets $V(H_1)=\{ 0 \}$ and $V(H_2)=\{ 1,2 \}$. The only edge cut assignment $\alpha$ such that $G^\alpha$ is even yields $\alpha(01)=\alpha(012)=12$, which results in $H^\alpha=H$. Following Remarks~\ref{rems}(\ref{rem:reduction}), we replace the edge $012$ first with $01$ (Iteration 1.1.1), and then with $02$ (Iteration 1.1.2).

{\em Iteration 1.1.1.} We have $H=\left(\ZZ_3,\mset{01,01,12} \right)$. In Step (1), we delete vertex 2, which is of degree 1, resulting in an edge of cardinality 1. Hence $H$ admits no Euler family.

{\em Iteration 1.1.2.} We have $H=\left(\ZZ_3,\{ 01,02,12 \} \right)$. This is an even graph, and so it admits an Euler family. Iteration 1.1 is successfully completed.

{\em Iteration 1.2.} The input hypergraph $H=\left(\{ 3,4 \},\{ 34 \} \right)$ is rejected in Step (1). Iteration 1 is completed unsuccessfully.

{\em Iteration 2.} We have $H^\alpha=\left(\ZZ_6,\{ 014, 012, 123,34 \} \right)$. There is a single non-empty connected component to be tested (Iteration 2.1).

{\em Iteration 2.1.} The input hypergraph is  $H=\left(\ZZ_5,\{ 014, 012, 123,34 \} \right)$. We choose a minimal edge cut $F=\{ 012,014 \}$. The hypergraph $H\sm F$ has connected components $H_1$ and $H_2$ with vertex sets $V(H_1)=\{ 0 \}$ and $V(H_2)=\{ 1,2,3,4 \}$. We first test the edge cut assignment $\alpha$ with $\alpha(014)=\alpha(012)=2$ (Iteration 2.1.1).

{\em Iteration 2.1.1.} We have $H^\alpha=\left(\{1,2,3,4\},\{ 12,14,123,34 \} \right)$. There is a single non-empty connected component to be tested (Iteration 2.1.1.1).

{\em Iteration 2.1.1.1.} The input hypergraph is  $H=\left(\{ 1,2,3,4\},\{ 12,14,123,34 \} \right)$. We choose a minimal edge cut $F=\{ 123,14 \}$. The hypergraph $H\sm F$ has components $H_1$ and $H_2$ with vertex sets $V(H_1)=\{ 1,2 \}$ and $V(H_2)=\{ 3,4 \}$. The only admissible edge cut assignment is $\alpha$ with $\alpha(123)=\alpha(14)=12$, which results in $H^\alpha=H$. Following Remarks~\ref{rems}(\ref{rem:reduction}), we replace the edge $123$ first with $13$ (Iteration 2.1.1.1.1), and then with $23$ (Iteration 2.1.1.1.2).

{\em Iteration 2.1.1.1.1.} The input hypergraph   $H=\left(\{ 1,2,3,4\},\{ 12,14,13,34 \} \right)$ gets rejected in Step (1).

{\em Iteration 2.1.1.1.2.} The input hypergraph is  $H=\left(\{ 1,2,3,4\},\{ 12,14,23,34 \} \right)$, which is an even graph. Iterations 2.1.1.1, 2.1.1, 2.1, and 2 are thus completed successfully.

We conclude that the original input hypergraph $H$ admits an Euler family.

Note that Algorithm~\ref{alg:ET} performed on $H$ runs very similarly, except that Iteration 1 is omitted because $H^\alpha$ has more than one non-empty connected component.
}
\phantom{a} \hfill \framebox[2.2mm]

\end{eg}

\color{black}


\section{Reduction Using Collapsed Hypergraphs}\label{sec:red2}

We shall now introduce our second main tool --- the collapsed hypergraph --- which allows for a binary-type of a reduction.

\begin{defn}{\rm
Let $H$ be a hypergraph, and $S$ a subset of its vertex set with $\emptyset \subsetneq S \subsetneq V(H)$. The {\em collapsed hypergraph of $H$ with respect to $S$} is the hypergraph $H \circ S=(V^{\circ}, E^{\circ})$ defined by
$$V^{\circ}=(V-S) \cup \{ u^{\circ} \} \quad \mbox{ and } \quad
E^{\circ}=\mset{e^{\circ}: e \in E(H), |e \cap (V-S)| \ge 1},$$
where
$$e^{\circ}= \left\{
\begin{array}{ll}
e & \mbox{ if } e \cap S = \emptyset \\
(e-S) \cup \{ u^{\circ} \} & \mbox{otherwise}
\end{array}
\right..$$
Here, $u^{\circ}$ is a new vertex, which we call the {\em collapsed vertex} of $H \circ S$.
}
\end{defn}

We are now ready for our second main reduction theorem, Theorem~\ref{thm:EF-collapsed}, which states that a hypergraph $H$ admits an Euler family if and only if some collapsed hypergraphs of $H$ (obtained using an edge cut assignment) admit Euler families.

\begin{thm}\label{thm:EF-collapsed}
Let $H$ be a connected hypergraph with a minimal edge cut $F$, and $\{ V_0, V_1 \}$ a partition of $V(H)$ into unions of the vertex sets of the connected components of $H \sm F$.

Then $H$ admits an Euler family if and only if
there exists an edge cut assignment $\alpha: F \to \ZZ_2^{[2]}$ such that $|\alpha^{-1}(01)|$ is even, and either
\begin{enumerate}[(i)]
\item $\alpha^{-1}(01) = \emptyset$ and $H^{\alpha}[V_i]$ has an Euler family for each $i \in \ZZ_2$, or
\item $\alpha^{-1}(01) \ne \emptyset$, and for each $i \in \ZZ_2$, the collapsed hypergraph $H^{\alpha} \circ V_i$ has an Euler family that traverses the collapsed vertex $u_i^{\circ}$ via each of the edges in $\{ f^{\circ}: f \in \alpha^{-1}(01) \}$.
\end{enumerate}
\end{thm}

\begin{proof}
Let $\F$ be an Euler family of $H$, and $\alpha: F \to \ZZ_2^{[2]}$ the edge cut assignment with $\alpha(f)=ij$ if and only if the edge $f$ is traversed by $\F$ via a vertex in $V_i$ and a vertex in $V_j$ (where $i=j$ is possible). By Lemma~\ref{lem:H^alpha}(i), the hypergraph $H^{\alpha}$ has an Euler family as well, say $\F^{\alpha}$. Since the only edges of $H^{\alpha}$ that intersect both $V_0$ and $V_1$ are edges of the form $f^{\alpha}$ with $f \in F$ and $\alpha(f)=01$,
it is easy to see that each closed trail in $\F^{\alpha}$ traverses an even number of such edges. Hence $|\alpha^{-1}(01)|$ is indeed even.

Suppose first that $\alpha^{-1}(01) = \emptyset$. Then $H^{\alpha}[V_i]$, for each $i \in \ZZ_2$, is a union of connected components of $H^{\alpha}$, and hence has an Euler family by Lemma~\ref{lem:decomposingfamilies}(ii).

Suppose now that $\alpha^{-1}(01) \ne \emptyset$. By symmetry, it suffices to show that $H^{\alpha} \circ V_1$ has an Euler family that traverses the collapsed vertex $u_1^{\circ}$ via each of the edges in $\{ f^{\circ}: f \in \alpha^{-1}(01) \}$. Let $T$ be a closed trail in the Euler family $\F^{\alpha}$ of $H^{\alpha}$ that traverses an edge $f^{\alpha}$, for some $f \in \alpha^{-1}(01)$. Then $T$ has the form
$$T=v_0^{(0)} f_0^{\alpha} u_0^{(1)} T_0^{(1)} v_0^{(1)} f_0'^{\alpha} u_0^{(0)} T_0^{(0)} v_1^{(0)} \ldots u_{k-1}^{(0)} T_{k-1}^{(0)} v_0^{(0)}$$
for some vertices $v_0^{(0)}, u_0^{(0)},v_1^{(0)}, u_1^{(0)},\ldots, u_{k-1}^{(0)} \in V_0$ and $u_0^{(1)},v_0^{(1)}, u_1^{(1)},v_1^{(1)}, \ldots, v_{k-1}^{(1)} \in V_1$; for some edges $f_0, f_0', f_1, f_1', \ldots, f_{k-1}' \in \alpha^{-1}(01)$; and for $i \in \ZZ_k$, for some $(u_i^{(0)},v_{i+1}^{(0)})$-trails $T_i^{(0)}$ in $H^{\alpha}[V_0]$ and $(u_i^{(1)},v_i^{(1)})$-trails $T_i^{(1)}$ in $H^{\alpha}[V_1]$.
Let $T^{\circ}$ be the sequence obtained from $T$ by replacing each subsequence of the form $v_i^{(0)} f_i^{\alpha} u_i^{(1)} T_i^{(1)} v_i^{(1)} f_i'^{\alpha} u_i^{(0)}$ with $v_i^{(0)} f_i^{\circ} u_1^{\circ}  f_i'^{\circ} u_i^{(0)}$. Then $T^{\circ}$ is a closed trail in $H^{\alpha} \circ V_1$, and it traverses the collapsed vertex  $u_1^{\circ}$ via each of the edges of the form $f^{\circ}$ for $f \in \alpha^{-1}(01)$ such that $T$ traverses $f^{\alpha}$. If we additionally define $T^{\circ}=T$ for all closed trails $T \in \F^{\alpha}$ that do not traverse any edges of the form $f^{\alpha}$, for some $f \in \alpha^{-1}(01)$, then it is clear that $\{ T^{\circ}: T \in \F^{\alpha} \}$ is a family of closed trails in $H^{\alpha} \circ V_1$ that jointly traverse each edge exactly once, and also traverse the collapsed vertex $u_1^{\circ}$ via each of the edges in $\{ f^{\circ}: f \in \alpha^{-1}(01) \}$.
To obtain an Euler family, we just need to concatenate all those closed trails in this family that traverse the collapsed vertex $u_1^{\circ}$.

To prove the converse, let $\alpha: F \to \ZZ_2^{[2]}$ be an edge cut assignment such that $|\alpha^{-1}(01)|$ is even, and either (i) or (ii) holds.

Suppose first that (i) holds. Since $\alpha^{-1}(01) = \emptyset$, for each $i \in \ZZ_2$, the hypergraph $H^{\alpha}[V_i]$ is a union of connected components of $H^{\alpha}$, and by Lemma~\ref{lem:decomposingfamilies}, since each of $H^{\alpha}[V_0]$ and $H^{\alpha}[V_1]$ admits an Euler family, so does $H^{\alpha}$. By Lemma~\ref{lem:H^alpha}(ii), it follows that $H$ admits an Euler family.

Suppose now that (ii) holds and $|\alpha^{-1}(01)|=2k$. For each $i \in \ZZ_2$, let $\alpha^{-1}(01)=\{ f_0^{(i)},  f_1^{(i)}, \ldots, f_{2k-1}^{(i)} \}$. By assumption, the collapsed hypergraph $H^{\alpha} \circ V_{1-i}$ admits an Euler family $\F^{(i)}$ that traverses the collapsed vertex $u_{1-i}^{\circ}$ via each of the edges $f^{\circ}$, for $f \in \alpha^{-1}(01)$. Let $T^{(i)}$ be the unique closed trail in $\F^{(i)}$ that traverses $u_{1-i}^{\circ}$. Then, without loss of generality, $T^{(i)}$ must be of the form
$$T^{(i)}=v_0^{(i)} \,\, (f_0^{(i)})^{\circ} \,\, u_{1-i}^{\circ} \,\, (f_1^{(i)})^{\circ} \,\, v_1^{(i)} \,\, T_1^{(i)} \,\, v_2^{(i)} \,\, (f_2^{(i)})^{\circ}  \ldots \,\, v_{2k-1}^{(i)} \,\, T_{k}^{(i)} \,\, v_0^{(i)} \label{eq:Ti}$$
for some vertices $v_0^{(i)}, v_1^{(i)}, \ldots, v_{2k-1}^{(i)} \in V_i$ and, for $j=1,2,\ldots,k$,  some $(v_{2j-1}^{(i)},v_{2j}^{(i)})$-trails $T_j^{(i)}$ in $H^{\alpha}[V_i]$. (Here, subscripts are evaluated modulo $2k$.)

Let $\pi: \ZZ_{2k} \to \ZZ_{2k}$ be a bijection such that $f_{\pi(j)}^{(1)}=f_j^{(0)}$, for all $j \in \ZZ_{2k}$. We thus have that $v_{\pi(j)}^{(1)} \in f_j^{(0)}$, for all $j \in \ZZ_{2k}$.

We now link the $(v_{2j-1}^{(0)},v_{2j}^{(0)})$-trails $T_j^{(0)}$ and $(v_{2j-1}^{(1)},v_{2j}^{(1)})$-trails $T_j^{(1)}$, for $j=1,2,\ldots,k$, into a family ${\cal T}$ of closed trails in $H^{\alpha}$ using the edges $(f_j^{(0)})^{\alpha}$, for $j \in \ZZ_{2k}$. In particular, the new closed trails will traverse each edge $(f_j^{(0)})^{\alpha}$ via anchors $v_j^{(0)}$ and $v_{\pi(j)}^{(1)}$.

Finally, let
$$\F= \left( \F^{(0)}-\{ T^{(0)} \} \right ) \cup \left( \F^{(1)}-\{ T^{(1)} \} \right ) \cup {\cal T}.$$
Since for each $i \in \ZZ_2$, the closed trails in $\F^{(i)}-\{ T^{(i)} \}$ traverse all edges of $H^\alpha[V_i]$ that are not traversed by $T^{(i)}$, and the closed trails in ${\cal T}$ traverse all edges of $H^\alpha[V_0] \cup H^\alpha[V_1]$ that are traversed by $T^{(0)}$ or $T^{(1)}$, as well as all edges in $\{ f^{\alpha}: f \in \alpha^{-1}(01) \}$, we conclude that $\F$
is an Euler family of $H^{\alpha}$.

It then follows by Lemma~\ref{lem:H^alpha}(ii) that $H$ admits an Euler family.
\end{proof}

Observe that in the proof of sufficiency in Theorem~\ref{thm:EF-collapsed} we have no control over the number of closed trails in the family ${\cal T}$; hence the analogous result for Euler tours does not hold in general. A weaker result is proved below: necessity is guaranteed by Corollary~\ref{cor:ET-collapsed-1}, while sufficiency is proved in Corollary~\ref{cor:ET-collapsed2} with an additional assumption on the edge cut assignment. This additional assumption, however, always holds  for edge cuts of cardinality at most 3.

\begin{cor}\label{cor:ET-collapsed-1}
Let $H$ be a connected hypergraph with a minimal edge cut $F$, and $\{ V_0, V_1 \}$ a partition of $V(H)$ into unions of the vertex sets of the connected components of $H \sm F$.

If $H$ admits an Euler tour, then there exists an edge cut assignment $\alpha: F \to \ZZ_2^{[2]}$ such that $|\alpha^{-1}(01)|$ is even, and either
\begin{enumerate}[(i)]
\item $\alpha^{-1}(01) = \emptyset$ and, without loss of generality, $H^{\alpha}[V_0]$ has an Euler tour and $H^{\alpha}[V_1]$ is empty, or
\item $\alpha^{-1}(01) \ne \emptyset$, and for each $i \in \ZZ_2$, the collapsed hypergraph $H^{\alpha} \circ V_i$ has an Euler tour that traverses the collapsed vertex $u_i^{\circ}$ via each of the edges in $\{ f^{\circ}: f \in \alpha^{-1}(01) \}$.
\end{enumerate}
\end{cor}

\begin{proof}
Let $T$ be an Euler tour of $H$, and $\alpha: F \to \ZZ_2^{[2]}$ the edge cut assignment with $\alpha(f)=ij$ if and only if the edge $f$ is traversed by $T$ via a vertex in $V_i$ and a vertex in $V_j$ (where $i=j$ is possible). We establish that $|\alpha^{-1}(01)|$ is even just as in the proof of necessity in Theorem~\ref{thm:EF-collapsed}.

By Lemma~\ref{lem:H^alpha}(i), the hypergraph $H^{\alpha}$ has an Euler tour as well. If $\alpha^{-1}(01) = \emptyset$, then $H^{\alpha}$ is disconnected. Hence without loss of generality, $H^{\alpha}[V_0]$ has an Euler tour and $H^{\alpha}[V_1]$ is empty.

Suppose now that $\alpha^{-1}(01) \ne \emptyset$. Following the proof of necessity in Theorem~\ref{thm:EF-collapsed}, we can show that for each $i \in \ZZ_2$, the Euler tour $T$ of $H$ gives rise to an Euler tour $T_i^{\circ}$ of $H^{\alpha} \circ V_i$ that traverses the collapsed vertex $u_i^{\circ}$ via each of the edges in $\{ f^{\circ}: f \in \alpha^{-1}(01) \}$.
\end{proof}

\begin{cor}\label{cor:ET-collapsed2}
Let $H$ be a connected hypergraph with a minimal edge cut $F$, and $\{ V_0, V_1 \}$ a partition of $V(H)$ into unions of the vertex sets of the connected components of $H \sm F$. Assume that there exists an edge cut assignment $\alpha: F \to \ZZ_2^{[2]}$ such that $|\alpha^{-1}(01)| \in \{ 0, 2 \}$, and either
\begin{enumerate}[(i)]
\item $\alpha^{-1}(01) = \emptyset$ and, without loss of generality, $H^{\alpha}[V_0]$ has an Euler tour and $H^{\alpha}[V_1]$ is empty, or
\item $\alpha^{-1}(01) \ne \emptyset$, and for each $i \in \ZZ_2$, the collapsed hypergraph $H^{\alpha} \circ V_i$ has an Euler tour that traverses the collapsed vertex $u_i^{\circ}$.
\end{enumerate}
Then $H$ admits an Euler tour.
\end{cor}

\begin{proof}
If $\alpha^{-1}(01) = \emptyset$, then $H^{\alpha} = H^{\alpha}[V_0] \cup H^{\alpha}[V_1]$, and since  $H^{\alpha}[V_1]$ is empty and $H^{\alpha}[V_0]$ admits an Euler tour, so does $H^{\alpha}$. By Lemma~\ref{lem:subhypergraph}(ii), it follows that $H$ admits an Euler tour.

The case $\alpha^{-1}(01) \ne \emptyset$ is similar to  the proof of sufficiency in Theorem~\ref{thm:EF-collapsed}, assuming Condition (ii).
We have $\alpha^{-1}(01)=\{ f_0^{(0)},  f_1^{(0)}\}= \{ f_0^{(1)},  f_1^{(1)}\}$.
By assumption, for each $i \in \ZZ_2$, the collapsed hypergraph $H^{\alpha} \circ V_{1-i}$ admits an Euler tour $T^{(i)}$ that traverses the collapsed vertex $u_{1-i}^{\circ}$, necessarily via the edges $f^{\circ}$, for $f \in \alpha^{-1}(01)$. Then $T^{(i)}$ must be of the form
$$T^{(i)}=v_0^{(i)} \,\, (f_0^{(i)})^{\circ} \,\, u_{1-i}^{\circ} \,\, (f_1^{(i)})^{\circ} \,\, v_1^{(i)} \,\, T_1^{(i)} \,\,  v_0^{(i)}$$
for some vertices $v_0^{(i)}, v_1^{(i)} \in V_i$ and a $(v_{1}^{(i)},v_{0}^{(i)})$-trail $T_1^{(i)}$ in $H^{\alpha}[V_i]$.

Since either $f_j^{(0)}= f_j^{(1)}$ or $f_j^{(0)}= f_{1-j}^{(1)}$ for $j \in \ZZ_2$, linking the $(v_{1}^{(0)},v_{0}^{(0)})$-trail $T_1^{(0)}$ and $(v_{1}^{(1)},v_{0}^{(1)})$-trail $T_1^{(1)}$  using the edges $(f_0^{(0)})^{\alpha}$ and $(f_1^{(0)})^{\alpha}$ clearly results in a closed trail $T$ in $H^{\alpha}$ that traverses all edges of $H^{\alpha}$ traversed by $T^{(0)}$ and $T^{(1)}$, as well as the edges $(f_0^{(0)})^{\alpha}$ and $(f_1^{(0)})^{\alpha}$.
We conclude that $T$ is an Euler tour of $H^{\alpha}$.

By Lemma~\ref{lem:H^alpha}(ii), we have that $H$ admits an Euler tour.
\end{proof}

Observe that in the case $|F| \le 3$, Corollary~\ref{cor:ET-collapsed2} gives a full converse to Corollary~\ref{cor:ET-collapsed-1}.

In Theorem~\ref{thm:EF-collapsed} and Corollaries~\ref{cor:ET-collapsed-1} and \ref{cor:ET-collapsed2}, we seem to be translating the problem of finding an Euler family (tour) to a different problem, namely, of finding an Euler family (tour) that traverses a specified vertex via each of the specified edges. In the next lemma, we show that an algorithm to find the former can be used to find the latter as well.

\begin{lem}\label{lem:fixed_vx}
Let $H$ be a hypergraph, $u \in V(H)$, and $F \subseteq E(H)$ such that each edge in $F$ is incident with the vertex $u$. Construct a hypergraph $H'$ from $H$ as follows: \textcolor{black}{for each edge $f \in F$ such that $|f| \ge 3$,}
\begin{itemize}
\item adjoin a new vertex $u_f$, and
\item replace $f$ with edges $f'=(f-\{u \}) \cup \{ u_f \}$ and $e_f=\{ u_f,u \}$.
\end{itemize}
Then $H$ has an Euler family (tour) traversing vertex $u$ via each edge in $F$ if and only if $H'$ has an Euler family (tour).
\end{lem}

\begin{proof}
\textcolor{black}{Let $F_{\ge 3}=\{ f \in F: |f| \ge 3 \}$.}

Assume $\F$ is an Euler family of $H$ traversing vertex $u$ via each edge in $F$. For each closed trail $T$ in $\F$, obtain a sequence $T'$ by replacing each subsequence of the form $fu$ in $T$ \textcolor{black}{(where $f \in F_{\ge 3}$)} with the subsequence $f' u_f e_f u$, and each subsequence of the form $uf$ \textcolor{black}{(where $f \in F_{\ge 3}$)} with the subsequence $u e_f u_f f'$. It is clear that $\F'=\{ T': T \in \F \}$ is an Euler family of $H'$.

Conversely, if $\F'$ is an Euler family of $H'$, then it must traverse each edge of the form $e_f$, \textcolor{black}{for  $f \in F_{\ge 3}$,} via either the subtrail $f' u_f e_f u$ or the subtrail $u e_f u_f f'$. For each closed trail $T' \in \F'$, obtain a sequence $T$ by replacing each subsequence of the form $f' u_f e_f u$ \textcolor{black}{(where $f \in F_{\ge 3}$)} with $fu$, and each subsequence of the form $u e_f u_f f'$ \textcolor{black}{(where $f \in F_{\ge 3}$)} with $uf$. It is then easy to see that $\F=\{ T: T' \in \F' \}$ is an Euler family of $H$ that traverses vertex $u$ via each edge in \textcolor{black}{$F_{\ge 3}$, as well as via each edge in $F-F_3$, as the edges in the latter set are of cardinality 2.}

Since in both parts of the proof we have $|\F|=|\F'|$, the statement for Euler tours holds as well.
\end{proof}

\bigskip

We shall now describe an algorithm based on Theorem~\ref{thm:EF-collapsed}, Lemma~\ref{lem:fixed_vx}, \textcolor{black}{ and Corollary~\ref{cor:2-edges}} that determines whether a hypergraph admits an Euler family.

\begin{alg}\label{alg:EF2} Does a connected hypergraph $H$ admit an Euler family? {\rm
\begin{enumerate}[(1)]
\item Sequentially delete vertices of degree at most 1 from $H$ until no such vertices remain or $H$ is trivial.
If at any step $H$ has an edge of cardinality less than 2, then $H$ has no Euler family --- exit.

\color{black}

\item Let $H^{[2]}$ be the graph obtained from $H$ by deleting all edges of cardinality greater than 2, and $G$ a maximal even subgraph of $H^{[2]}$. Delete all edges of $G$ from $H$.

\item Repeat Steps (1) and (2) as needed.
\item If $H$ is a graph, then it has an Euler family if and only if it is even --- exit.

\color{black}
\item Find a minimal edge cut $F$.
\item Let $H_i$, for $i \in I$, be the connected components of $H \sm F$.
\item Find a bipartition $\{ V_0, V_1 \}$ of $V(H)$ into unions of $V(H_i)$, for $i \in I$.
\item For all edge cut assignments $\alpha: F \to \ZZ_2^{[2]}$: \label{step-EF}
\begin{enumerate}
\item If $|\alpha^{-1}(01)|$ is odd --- discard this $\alpha$.
\item If $\alpha^{-1}(01) = \emptyset$: if $H^{\alpha}[V_0]$ and $H^{\alpha}[V_1]$ both have Euler families {\em (recursive call)}, then $H$ has an Euler family --- exit.
\item If $\alpha^{-1}(01) \ne \emptyset$: if $H^{\alpha} \circ V_1$ and $H^{\alpha}\circ V_0$ both have Euler families {\em (recursive call)} that traverse the collapsed vertex via each of the edges of the form $f^{\circ}$, for $f \in \alpha^{-1}(01)$, then $H$ has an Euler family --- exit.
\end{enumerate}
\item $H$ has no Euler family --- exit.
\end{enumerate}
}
\end{alg}

The analogous algorithm for Euler tours, Algorithm~\ref{alg:ET2} below, relies on Corollary~\ref{cor:ET-collapsed2} and Lemma~\ref{lem:fixed_vx}, \textcolor{black}{as well as Theorem~\ref{thm:ET-J} and Corollaries~\ref{cor:cut-edge} and \ref{cor:2-edges}}. If the hypergraph has no edge cut satisfying the assumptions of Corollary~\ref{cor:ET-collapsed2},  then Algorithm~\ref{alg:ET} is invoked.

\begin{alg}\label{alg:ET2} Does a connected hypergraph $H$ admit an Euler tour?
{\rm
\begin{enumerate}[(1)]
\item Sequentially delete vertices of degree at most 1 from $H$ until no such vertices remain or $H$ is trivial.
If at any step $H$ has an edge of cardinality less than 2, then $H$ is not eulerian --- exit.

\color{black}
\item Let $H^{[2]}$ be the graph obtained from $H$ by deleting all edges of cardinality greater than 2, and $G$ a maximal even subgraph of $H^{[2]}$. For each non-empty connected component $G_1$ of $G$, replace $E(G_1)$ in $H$ with the edge set of any cycle on $V(G_1)$.

\item Repeat Steps (1) and (2) as needed.

\item If $H$ is a graph, then it is eulerian if and only if it is even and has at most one non-trivial connected component --- exit.

\color{black}
\item Find a minimum edge cut $F$. If $|F|=1$, then $H$ is not eulerian --- exit. \label{step:min-cut}
\item Let $H_i$, for $i \in I$, be the connected components of $H \sm F$.

\color{black}
\item Let $J = \{ i \in I: H_i \mbox{ is non-empty} \}$. If $|F|<|J|$, then $H$ is not eulerian --- exit.

\color{black}
\item \label{step-main} For all bipartitions  $\{ V_0, V_1 \}$ of $V(H)$ into unions of $V(H_i)$, for $i \in I$   \\
    and for all edge cut assignments $\alpha: F \to \ZZ_2^{[2]}$:
\begin{enumerate}
\item If $|\alpha^{-1}(01)|$ is odd --- discard this $\alpha$.
\item If $|\alpha^{-1}(01)| = 0$: if $H^{\alpha}[V_0]$ is empty and $H^{\alpha}[V_1]$ has an Euler tour, or vice-versa {\em (recursive call)}, then $H$ has an Euler tour --- exit.
\item If $|\alpha^{-1}(01)|=2$: if $H^{\alpha} \circ V_1$ and $H^{\alpha}\circ V_0$ both have an Euler tour  that traverses the collapsed vertex {\em (recursive call)}, then $H$ has an Euler tour --- exit.
\end{enumerate}
\item Use Algorithm~\ref{alg:ET}. \label{step:hope-not}
\end{enumerate}
}
\end{alg}

\begin{rems}\label{rems:2}{\rm
\begin{enumerate}[(i)]
\item Remarks~\ref{rems}(i) and \ref{rems}(vii) apply to Algorithms~\ref{alg:EF2} and \ref{alg:ET2} as well.
\item \label{rem:min2} In Step~(\ref{step:min-cut}) of Algorithm~\ref{alg:ET2}, we chose to find a minimum edge cut --- see Remark~\ref{rems}(\ref{rem:min}) --- in the hope of finding an edge cut of cardinality at most 3; in this case, a reduction using Corollary~\ref{cor:ET-collapsed2} is guaranteed, and Step~(\ref{step:hope-not}) is not reached in this call.

\color{black}
\item As in Algorithms~\ref{alg:EF} and \ref{alg:ET} --- see Remarks~\ref{rems}(\ref{rem:reduction}) --- a few adjustments may be necessary to guarantee that Algorithms~\ref{alg:EF2} and \ref{alg:ET2} indeed terminate.

\color{black}
    For any hypergraph $H=(V,E)$, let $|H|_{\ge 3}=\sum_{e \in E, |e|\ge 3} |e|$ and $|H|_{2}=\sum_{e \in E, |e|=2} |e|$, and call $|H|_{\ge 3}$ the {\em relevant size} of $H$. Since the problem of existence of an Euler family (tour) is computationally easy for graphs (Step (4) in both algorithms), for either algorithm to terminate, it suffices that each recursive call be performed on a hypergraph $H'$ with $|H'|<|H|$ or $|H'|_{\ge 3}<|H|_{\ge 3}$. (Note that the use of Lemma~\ref{lem:fixed_vx} in Step (8c) may result in $|H'|_2>|H|_2$, but as long as $|H'|_{\ge 3}<|H|_{\ge 3}$, we are still moving in the right direction.)
    To guarantee such a relevant reduction, we first proceed as described in Remarks~\ref{rems}(\ref{rem:reduction}), that is, favouring edge cut assignments $\alpha$ that yield $|H^\alpha|_{\ge 3}< |H|_{\ge 3}$, and therefore $|H'|_{\ge 3}<|H|_{\ge 3}$ for all recursive calls.

\color{black}
    However, we may need to examine an edge cut assignment $\alpha$ that yields $|H^\alpha|_{\ge 3}= |H|_{\ge 3}$. In this case, we know that $\alpha(f)=12$ for all $f \in F$, and the recursive call involves $H'=H^{\alpha} \circ V_i$ (for $i \in \{ 0,1 \}$). Clearly, $|H'|_{\ge 3} \le |H|_{\ge 3}$, but suppose that $|H'|_{\ge 3} = |H|_{\ge 3}$ and $|H'| \ge |H|$ (in other words, $|H'|_{2} \ge |H|_{2}$). This situation may arise when Lemma~\ref{lem:fixed_vx} has been invoked or $|(H\sm F)[V_i]|$ is small.

\color{black}
    If there exists $f \in F$ such that $|f| \ge 3$, then $|f^\circ|=|f|$ and $f \cap V_i=\{ x \}$ for some vertex $x \in V_i$. For each $y \in f \cap V_{1-i}$, we then let $e_y=\{x,y\}$ and $H_y'=(H' \sm f) + \{ e_y \}$. Note that $|H_y'|_{\ge 3} < |H|_{\ge 3}$, and it is easy to see that $H'$ admits an Euler family (tour) if and only if at least one such hypergraph $H_y'$ does. Hence, if necessary, we test all such hypergraphs $H_y'$.

\color{black}
    Otherwise, it must be that $|f|=2$ for all $f \in F$, and $|H'| = |H|$ as Lemma~\ref{lem:fixed_vx} was not invoked. Fix any $f^\ast \in F$, and let $f^\ast \cap V_{1-i}=\{ x \}$. For each $f \in F-\{ f^\ast\}$, we then let $f \cap V_{1-i}=\{ y_f \}$, $e_f=\{x,y_f\}$, and $H_f'=(H' \sm \{f^\ast, f\}) + \{ e_f \}$. Note that $|H_f'| < |H'|=|H|$, and again it is easy to see that $H'$ admits an Euler family (tour) if and only if at least one such hypergraph $H_f'$ does.

\color{black}
\item \label{rem:complexity} \textcolor{black}{In Step (8) of Algorithm~\ref{alg:ET2}, we examine all bipartitions of $V(H)$ to maximize the chance of finding an edge cut assignment $\alpha$ such that $|\alpha^{-1}(01)| \in \{ 0,2 \}$.} Observe that for a given edge cut $F$ such that $H \sm F$ has exactly $|I|$ connected components, the number of bipartitions $\{ V_0, V_1 \}$ is $2^{|I|-1}-1$, and the number of edge cut assignments $\alpha: F \to \ZZ_2^{[2]}$ is at most $3^{|F|}$. Hence the loop at Step~(\ref{step-main}) of Algorithm~\ref{alg:ET2} will be executed at most $(2^{|I|-1}-1) \cdot 3^{|F|}$ times.
\item \label{rem:bipartition} For suitable bipartitions $\{ V_0, V_1 \}$, we expect that the reduction at Step~(\ref{step-EF}) in both algorithms is the most efficient (that is, the number of recursive calls is minimized) if $|V_0|$ and $|V_1|$ are as equal as possible.
\end{enumerate}}
\end{rems}

\color{black}

\begin{figure}[t]
\centerline{\includegraphics[scale=0.8]{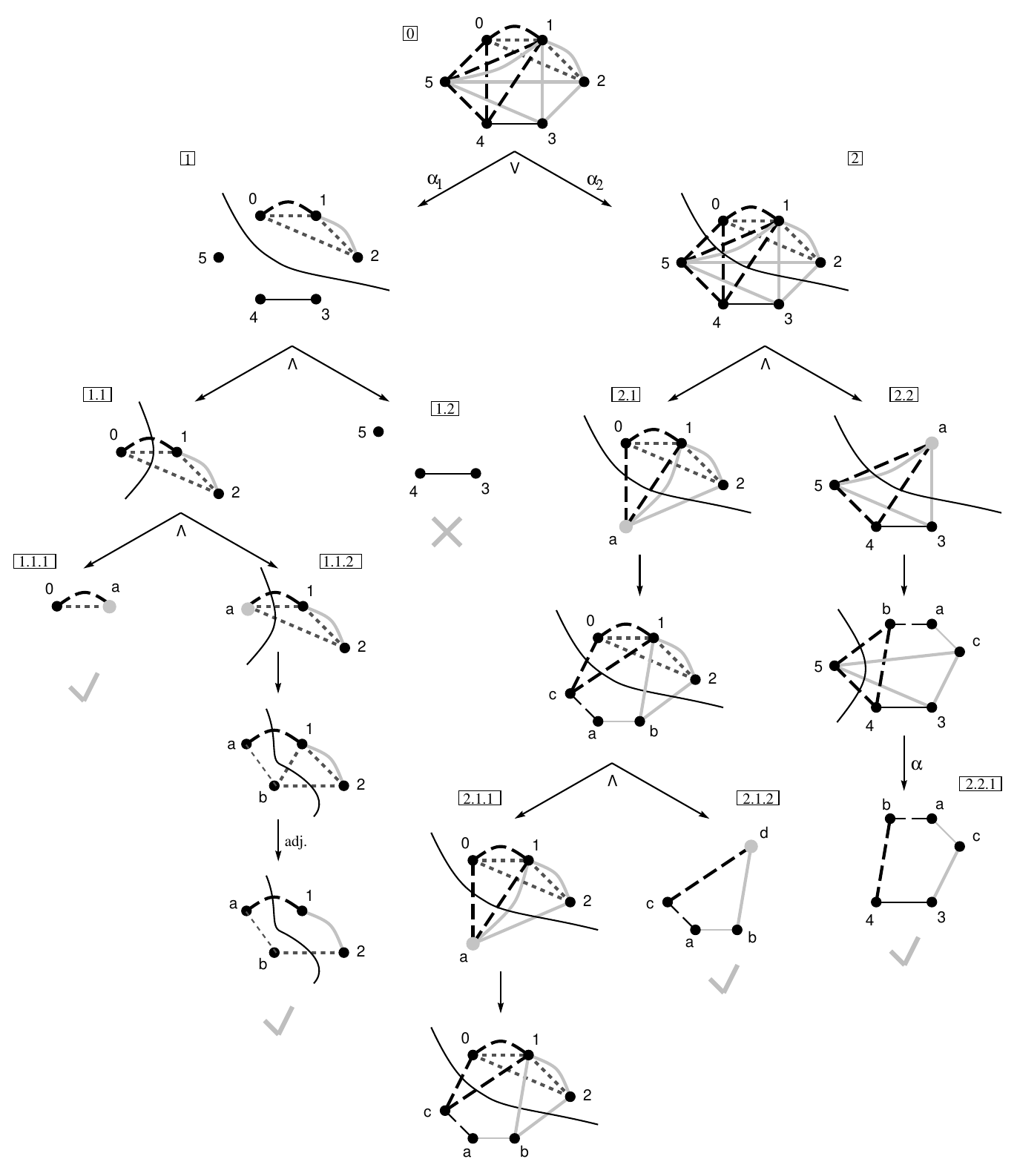}}
\caption{Illustrating Example~\ref{eg:Alg2} (continued in Figure~\ref{fig:Fig4}).}\label{fig:Fig3}
\end{figure}

\begin{figure}[t]
\centerline{\includegraphics[scale=0.78]{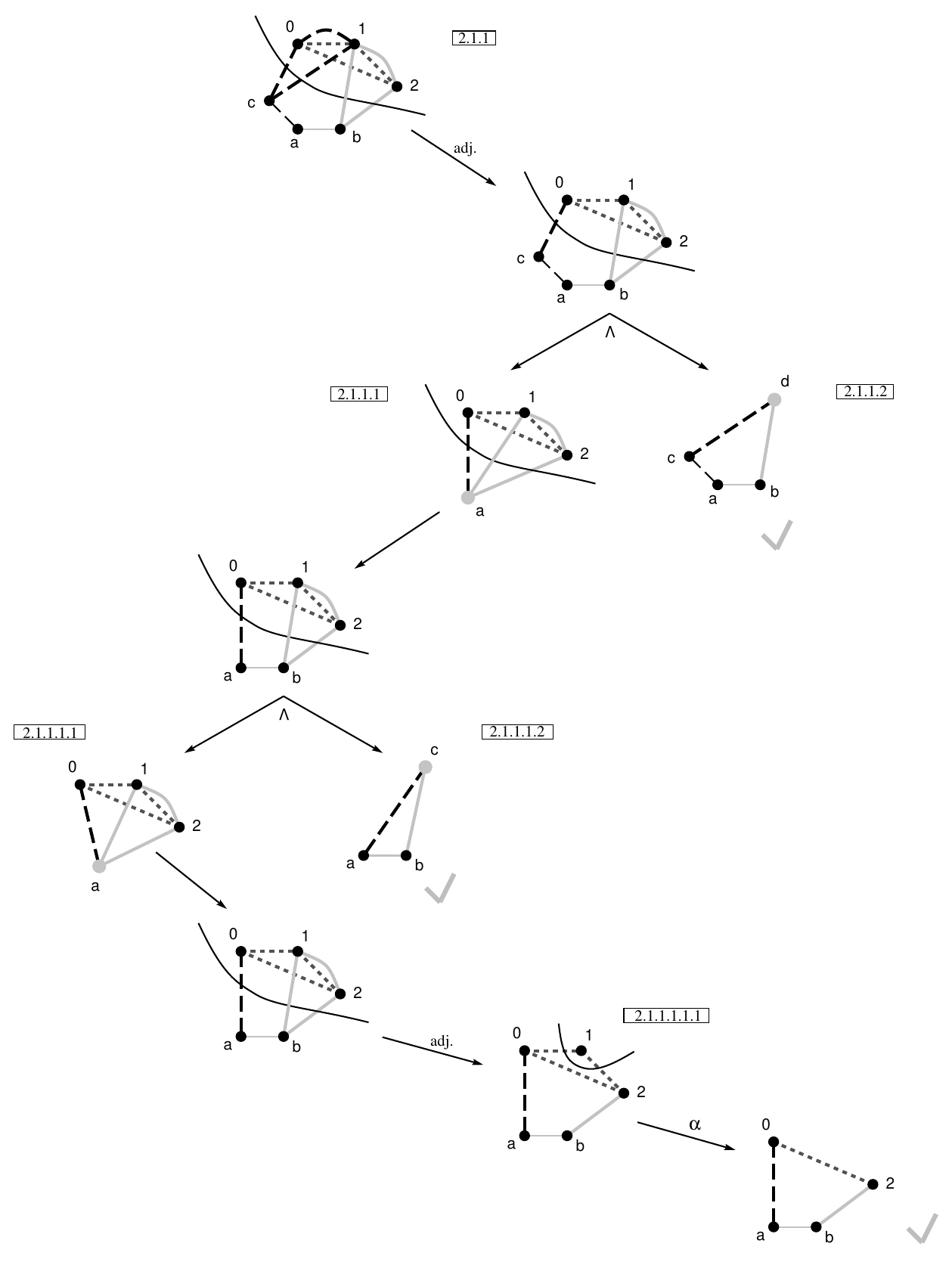}}
\caption{Illustrating Example~\ref{eg:Alg2}  (continued from Figure~\ref{fig:Fig3}).}\label{fig:Fig4}
\end{figure}

\begin{eg}\label{eg:Alg2}{\rm
We shall now demonstrate Algorithm~\ref{alg:EF2} using the same hypergraph as in Example~\ref{eg:Alg1}; that is, hypergraph $H=(V,E)$ with $V=\ZZ_6$ and $E=\{ 012, 0145,1235,34 \}$. The main hypergraphs in the procedure are shown in Figures~\ref{fig:Fig3} and \ref{fig:Fig4}. We use the same symbols as in Figure~\ref{fig:Fig2}; in addition, the collapsed vertices are shown in grey, and new vertices are labelled $a,b,c, \ldots$. New edges arising from the use of Lemma~\ref{lem:fixed_vx} are drawn in the same style as their parent edges, just thinner. Note that some minor details are omitted for brevity.

{\em Iteration 0.} The input hypergraph is  $H=\left(\ZZ_6,\{ 012, 0145,1235,34 \} \right)$, and as in Example~\ref{eg:Alg1}, we start by choosing a minimal edge cut $F=\{ 0145,1235 \}$, so $H\sm F$ has connected components $H_1, H_2$, and $H_3$ with vertex sets $V(H_1)=\ZZ_3$, $V(H_2)=\{ 3,4 \}$, and $V(H_3)=\{ 5 \}$. We then choose the bipartition $\{ V_0,V_1 \}$ with $V_0=\ZZ_3$ and $V_1=\{ 3,4,5 \}$, and  first examine the edge cut assignment $\alpha$ with $\alpha(0145)=\alpha(1235)=0$ (Iteration 1), and then $\alpha(0145)=\alpha(1235)=01$ (Iteration 2).

{\em Iteration 1.} We have $H^\alpha=\left(\ZZ_6,\{ 01, 012, 12,34 \} \right)$. As $\alpha^{-1}(01) = \emptyset$, we need to test its subhypergraphs $H^\alpha[V_0]$ and $H^\alpha[V_1]$ (Iterations 1.1 and 1.2, respectively).

{\em Iteration 1.1.} The input hypergraph is  $H=\left(\ZZ_3,\{ 01,012, 12 \} \right)$. We choose a minimal edge cut $F=\{ 01,012 \}$. The hypergraph $H\sm F$ has connected components $H_1$ and $H_2$, with vertex sets $V(H_1)=\{ 0 \}$ and $V(H_2)=\{ 1,2 \}$, so our bipartition is $\{ V_0,V_1 \}$ with $V_0=\{ 0 \}$ and $V_1=\{ 1,2 \}$. The only edge cut assignment $\alpha$ with $|\alpha^{-1}(01)|$ even is defined by $\alpha(01)=\alpha(012)=01$, which yields $H^\alpha=H$. As $\alpha^{-1}(01) \ne \emptyset$, we need to test $H^\alpha \circ V_0$ (Iteration 1.1.1) and $H^\alpha \circ V_1$ (Iteration 1.1.2).

{\em Iteration 1.1.1.} We have $H=\left(\{ 0,a \},\mset{ 0a,0a } \right)$. This is an even graph, so it admits an Euler family traversing the collapsed vertex $a$.

{\em Iteration 1.1.2.} We have $H=\left(\{ 1,2,a\},\{ 12,12a,1a \} \right)$, with $a$ being the collapsed vertex. To ensure its traversal, we use Lemma~\ref{lem:fixed_vx}, modifying the hypergraph to $H'=\left(\{ 1,2,a,b\},\{ 12,12b,1a,ab \} \right)$. As the size of the hypergraph has increased from Iteration 1.1, while the relevant size remained the same, we replace the edge $12b$ with $2b$ (see Remarks~\ref{rems:2}(\ref{rem:reduction})), which results in an even graph. This successfully completes Iteration 1.1.

{\em Iteration 1.2.} The input hypergraph $H=\left(\{ 3,4,5 \},\{ 34 \} \right)$ is rejected in Step (1). Iteration 1 is thus completed, but unsuccessfully.

{\em Iteration 2.} We have $H^\alpha=H$. As $|\alpha^{-1}(01)| =2$, we need to test $H^\alpha \circ V_0$ (Iteration 2.1) and $H^\alpha \circ V_1$ (Iteration 2.2).

{\em Iteration 2.1.} The input hypergraph is  $H=\left(\{ 0,1,2,a \},\{ 012,01a,12a \} \right)$, with $a$ being the collapsed vertex. To ensure its traversal, we use Lemma~\ref{lem:fixed_vx}, modifying the hypergraph to $H'=\left(\{ 0,1,2,a,b,c\},\{ 012,01c,12b,ab,ac \} \right)$.
We choose a minimal edge cut $F=\{ 01c,12b \}$. The hypergraph $H\sm F$ has two connected components, yielding a single bipartition $\{ V_0,V_1 \}$, with $V_0=\{ 0,1,2 \}$ and $V_1=\{ a,b,c \}$.
We choose the edge cut assignment $\alpha$ with $\alpha(01c)=\alpha(12b)=01$, which results in $H^\alpha=H$.
As $|\alpha^{-1}(01)| =2$, we need to test $H^\alpha \circ V_0$ (Iteration 2.1.1) and $H^\alpha \circ V_1$ (Iteration 2.2.2).

\FloatBarrier

{\em Iteration 2.1.1.} The input hypergraph is (again) $H=\left(\{ 0,1,2,a \},\{ 012,01a,12a \} \right)$, with $a$ being the newly collapsed vertex. To ensure its traversal, we use Lemma~\ref{lem:fixed_vx}, modifying the hypergraph to $H'=\left(\{ 0,1,2,a,b,c\},\{ 012,01c,12b,ab,ac \} \right)$. We observe that this hypergraph $H'$ is identical to $H'$ from Iteration 2.1. Hence we perform an adjustment as advised in Remarks~\ref{rems:2}(\ref{rem:reduction}), replacing the edge $01c$ with $0c$.

We then choose a minimal edge cut $F=\{ 0c,12b \}$,
resulting in the bipartition $\{ V_0,V_1 \}$, with $V_0=\{ 0,1,2 \}$ and $V_1=\{ a,b,c \}$.  The only edge cut assignment $\alpha$ with
$|\alpha^{-1}(01)|$ even is defined by $\alpha(0c)=\alpha(12b)=01$, which yields $H^\alpha=H$. As $\alpha^{-1}(01) \ne \emptyset$, we need to test $H^\alpha \circ V_0$ (Iteration 2.1.1.1) and $H^\alpha \circ V_1$ (Iteration 2.1.1.2).

{\em Iteration 2.1.1.1.} The input hypergraph is $H=\left(\{ 0,1,2,a \},\{ 012,0a,12a \} \right)$, with $a$ being the newly collapsed vertex. To ensure its traversal, we use Lemma~\ref{lem:fixed_vx}, modifying the hypergraph to $H'=\left(\{ 0,1,2,a,b\},\{ 012,0a,12b,ab \} \right)$. We choose a minimal edge cut $F=\{ 0a,12b \}$.
The only possible bipartition is $\{ V_0,V_1 \}$, with $V_0=\{ 0,1,2 \}$ and $V_1=\{ a,b \}$, and the only edge cut assignment $\alpha$ with
$|\alpha^{-1}(01)|$ even is defined by $\alpha(0a)=\alpha(12b)=01$,  resulting in $H^\alpha=H$. We thus need to test $H^\alpha \circ V_0$ (Iteration 2.1.1.1.1) and $H^\alpha \circ V_1$ (Iteration 2.1.1.1.2).

{\em Iteration 2.1.1.1.1.} The input hypergraph is (again) $H=\left(\{ 0,1,2,a \},\{ 012,0a,12a \} \right)$, with $a$ being the newly collapsed vertex.  To ensure its traversal, we modify the hypergraph to $H'=\left(\{ 0,1,2,a,b\},\{ 012,0a,12b,ab \} \right)$, and observe that this hypergraph $H'$ is identical to $H'$ from Iteration 2.1.1.1. Hence we perform an adjustment (see Remarks~\ref{rems:2}(\ref{rem:reduction})), replacing the edge $12b$ with $2b$.

{\em Iteration 2.1.1.1.1.1.} The input hypergraph is $H=\left(\{ 0,1,2,a,b \},\{ 012,0a,2b,ab \} \right)$. In Step (1), vertex 1 is deleted, resulting in the even graph $H'=\left(\{ 0,2,a,b \},\{ 02,0a,2b,ab \} \right)$. This successfully completes Iteration 2.1.1.1.1.

{\em Iteration 2.1.1.1.2.} The input hypergraph is $H=\left(\{ a,b,c \},\{ ab,ac,bc \} \right)$, with $c$ being the newly collapsed vertex. As $H$ is an even graph, it admits an Euler family that traverses the collapsed vertex. Iteration 2.1.1.1 is thus successfully completed.

\FloatBarrier

{\em Iteration 2.1.1.2.} We have $H=\left(\{ a,b,c,d \},\{ ab,ac,bd,cd \} \right)$, where $d$ is the newly collapsed vertex. Again, $H$ is an even graph, thus admitting an Euler family that traverses the collapsed vertex. Hence Iteration 2.1.1 is successfully completed.

{\em Iteration 2.1.2.} Identical to Iteration 2.1.1.2., successfully completing Iteration 2.1.

{\em Iteration 2.2.} The input hypergraph is  $H=\left(\{ 3,4,5,a \},\{ 34,35a,45a \} \right)$, with $a$ being the collapsed vertex. To ensure its traversal, we first modify the hypergraph to $H'=\left(\{ 3,4,5,a,b,c \},\{ 34,35c,45b,ab,ac \} \right)$.
We then choose a minimal edge cut $F=\{ 35c,45b \}$. The hypergraph $H\sm F$ has three connected components, and we choose a bipartition $\{ V_0,V_1 \}$, with $V_0=\{ 3,4,a,b,c \}$ and $V_1=\{ 5 \}$.
Furthermore, we choose the edge cut assignment $\alpha$ with $\alpha(35c)=\alpha(45b)=0$. As $\alpha^{-1}(01) = \emptyset$ and $H^\alpha[V_1]$ is empty, we only need to test $H^\alpha[V_0]$ (Iteration 2.2.1).

{\em Iteration 2.2.1.} We have $H=\left(\{ 3,4,a,b,c \},\{ 34,3c,4b,ab,ac \} \right)$, which is an even graph. This successfully completes Iterations 2.2 and 2.

We conclude that the original input hypergraph $H$ admits an Euler family.

Note that Algorithm~\ref{alg:ET2} performed on $H$ runs very similarly, except that Iteration 1 is omitted because $H^\alpha$ has more than one non-empty connected component.
}
\end{eg}

\color{black}

\color{black}

\section{Conclusions}

\color{black}
In this paper, we proved several results showing that, using an edge cut of a hypergraph $H$, the problem of existence of an Euler family (tour) in $H$ can be reduced to the analogous problem in smaller hypergraphs derived from $H$. These results led to Algorithms~\ref{alg:EF} and \ref{alg:EF2} for finding Euler families, and Algorithms~\ref{alg:ET} and \ref{alg:ET2} for finding Euler tours. We shall now compare the two algorithms in each pair.

\color{black}
\begin{rems}\label{rems:3}{\rm Comparison of Algorithms~\ref{alg:EF} and \ref{alg:EF2}.
\begin{enumerate}[(i)]
\item \label{rem:alpha} In the worst case, in Algorithms~\ref{alg:EF} and \ref{alg:EF2}, we need to check all edge cut assignments $\alpha: F \to I^{[2]}$ and  $\alpha: F \to \ZZ_2^{[2]}$, respectively. Since $|I| \ge |\ZZ_2|$, this loop will require at least as many iterations in Algorithm~\ref{alg:EF} as in Algorithm~\ref{alg:EF2}, and possibly many more.
\color{black}
\item In the same loop, the number of recursive calls in Algorithm~\ref{alg:EF} can be anywhere between 1 and $|I|$, inclusive, while Algorithm~\ref{alg:EF2} requires 2 recursive calls.
\color{black}
\item \label{rem:recursive-size} The size of the input hypergraph $H'$ for a recursive call in Algorithm~\ref{alg:EF} depends heavily on $F$ and $\alpha$, and would be hard to estimate. In Algorithm~\ref{alg:EF2}, however, the expected size is roughly $\frac{1}{2}|H|$.
\color{black}
\item Constructing an arbitrary bipartition $\{ V_0,V_1\}$ in Algorithm~\ref{alg:EF2} takes a negligible amount of time; however, if we wish to find a bipartition that, as much as possible, equalizes $|V_0|$ and $|V_1|$  --- see Remark~\ref{rems:2}(\ref{rem:bipartition}) --- then we may need to search through all $2^{|I|-1}-1$ bipartitions, and this step may increase the time complexity of Algorithm~\ref{alg:EF2}.
\color{black}
\item \label{rem:H'} In Algorithm~\ref{alg:EF}, the recursive calls use as the input the connected components of $H^\alpha$, while in Algorithm~\ref{alg:EF2}, the two collapsed hypergraphs need to be constructed first, and then (if $F$ contains edges of cardinality at least 3) further modified using Lemma~\ref{lem:fixed_vx}. These operations, however, do not increase time complexity.
\color{black}
\item The time complexities of Algorithms~\ref{alg:EF} and \ref{alg:EF2} can be compared in a similar way, with very similar results, as in Remark~\ref{rems:4}(\ref{rem:compl-T}) below for Algorithms~\ref{alg:ET} and \ref{alg:ET2}.
\end{enumerate}}
\end{rems}
\color{black}
\begin{rems}\label{rems:4}{\rm Comparison of Algorithms~\ref{alg:ET} and \ref{alg:ET2}.
\begin{enumerate}[(i)]
\item The main disadvantage of Algorithm~\ref{alg:ET2} is that it may have to resort to using Algorithm~\ref{alg:ET} if $H$ has no edge cut of cardinality at most 3.
\color{black}
\item The second disadvantage of Algorithm~\ref{alg:ET2} is that it uses an external algorithm to find a minimum edge cut; see Remark~\ref{rems:2}(\ref{rem:min2}). As mentioned in Remark~\ref{rems}(\ref{rem:min}), efficient algorithms for this problem do exist \cite{CheXu}.
\color{black}
\item In the worst case, the main loop in Algorithm~\ref{alg:ET}  requires ${{|I|+1} \choose 2}^{|F|}$ iterations  (see Remark~\ref{rems}(\ref{rem:n-alphas})), while the main loop in Algorithm~\ref{alg:ET2} requires $(2^{|I|-1}-1)\cdot 3^{|F|}$ iterations (see Remark~\ref{rems:2}(\ref{rem:complexity})). If $|I|=2$, then the two numbers are the same; however, assuming $|F|$ is bounded, the number of iterations in Algorithm~\ref{alg:ET2} grows much faster with $|I|$.
\color{black}
\item In the same loop, Algorithm~\ref{alg:ET} requires 1 recursive call, while Algorithm~\ref{alg:ET2} requires 1 or 2 recursive calls, depending on whether Step (8b) or (8c) applies, respectively.
\color{black}
\item In Algorithm~\ref{alg:ET}, a recursive call happens only with a hypergraph $H'$ that is the unique non-trivial connected component of $H^\alpha$, hence its size is close to the size of $|H|$ (but we expect that the number of edge cut assignments $\alpha$ requiring a recursive call will be small). In Algorithm~\ref{alg:ET2}, the expected size of $H'$ is either very close to $|H|$ (if Step (8b) applies) or approximately $\frac{1}{2}|H|$ (if Step (8c) applies).

\color{black}
\item \label{rem:compl-T} For an overall comparison of time complexities, let $\tau(p)$ and $\sigma(p)$ denote the time complexity functions of Algorithms~\ref{alg:ET} and \ref{alg:ET2}, respectively, where $p=|H|$. To facilitate the comparison, we shall assume that at any step of the algorithm (including recursive calls) we have $1 \le |F| \le k$ and $2 \le |I| \le m$ for some constants $k$ and $m$, and that in any recursive call, we have $|H'| \le \frac{p}{c}$ for some constant $c > 1$. In addition, we assume that Step~(\ref{step:hope-not}) of Algorithm~\ref{alg:ET2} is never reached.
    With these assumptions, we have
    $$\tau(p) \approx  {m+1 \choose 2}^{k} \cdot \tau \left(\frac{p}{c} \right) \approx m^{2k} \cdot \tau \left(\frac{p}{c} \right)
    \approx  (m^{2k})^{\log_{c} p} = p^{\log_{c} (m^{2k})}$$
    and
    $$\sigma(p) \approx 2^m 3^k \cdot  \sigma \left( \frac{p}{c} \right) \approx  (2^m 3^k)^{\log_{c} p} = p^{\log_{c} (2^m 3^k)}.$$
    Thus, $\tau(p)$ and $\sigma(p)$ are both polynomial in $p$; however, which of the polynomial orders is smaller depends on the relationship between $k$ and $m$. Roughly speaking, $\log_{c} (2^m 3^k)=m \log_c 2+k \log_c 3$ is smaller when $k$ is larger than $m$, while $\log_{c} (m^{2k})=2k \log_c m$ is smaller when $m$ is much larger than $k$.
\end{enumerate}}
\end{rems}

\color{black}

\bigskip\bigskip

\centerline{\bf Acknowledgement}

\medskip

The first author gratefully acknowledges support by the Natural Sciences and Engineering Research Council of Canada (NSERC), Discovery Grant RGPIN-2022-02994.

\end{document}